\documentclass[11pt,reqno]{amsart}

\title[Combination theorem: Amalgams]{Klein--Maskit combination theorem for Anosov subgroups: Amalgams}
\author{Subhadip Dey}
\address{Max Planck Institute for Mathematics in the Sciences in Leipzig,
    Inselstraße 22,
    04103 Leipzig,
    Germany
}
\email{subhadip.dey@mis.mpg.de}
\author{Michael Kapovich}
\address{Department of Mathematics,
   University of California, Davis,
   One Shields Ave, Davis, CA 95616}
\email{kapovich@ucdavis.edu}
\date{August 6, 2024}


\usepackage{graphicx}
\usepackage[dvipsnames]{xcolor}
\usepackage{enumerate,overpic,caption,amssymb}
\usepackage[linktocpage]{hyperref}
\hypersetup{
    colorlinks=true,
    linkcolor=blue,
    hypertexnames=false,
    filecolor=black,      
    urlcolor=black,
    citecolor=blue
}

\numberwithin{equation}{section} 

\usepackage{cleveref}
\newtheorem{theorem}{Theorem}[section]
\newtheorem{proposition}[theorem]{Proposition}
\newtheorem{lemma}[theorem]{Lemma}
\newtheorem{corollary}[theorem]{Corollary}
\newtheorem*{claim}{Claim}
\newtheorem{maintheorem}{Theorem}

\theoremstyle{remark}
\newtheorem{remark}[theorem]{Remark}
\newtheorem{notation}[theorem]{Notation}
\newtheorem{example}[theorem]{Example}
\theoremstyle{definition}
\newtheorem{definition}[theorem]{Definition}
\newtheorem{assumption}[theorem]{Assumption}
\newtheorem{case}{Case}

\def\C{\mathbf{C}}
\def\H{\mathbb{H}}
\def\N{\mathbf{N}}
\def\R{\mathbf{R}}
\def\Z{\mathbf{Z}}

\def\<{\langle}
\def\>{\rangle}
\def\acts{\curvearrowright}
\def\a{\alpha}
\def\b{\beta}
\def\e{\varepsilon}
\def\E{\mathcal{E}}
\def\G{\Gamma}
\def\bG{{\overline\Gamma}}
\def\g{\gamma}
\def\tg{{\tilde\g}}

\def\hb{\hat\beta}

\def\hg{\hat\gamma}
\def\hp{{\hat\mu}}
\def\ho{\hat\omega}
\def\H{\mathrm H}
\def\HH{\mathbb{H}}

\def\Fn{\mathrm{Flag}({\nmod})}
\def\Fs{\mathrm{Flag}({\sigma_{\rm mod}})}
\def\Ft{\mathrm{Flag}({\tau_{\rm mod}})}
\def\nmod{\nu_{\rm mod}}
\def\o{\circ}
\def\P{{\mathrm M}}

\def\PP{\mathbf{P}}
\def\Pf{{{\mathrm M}\star_{\phi}}}
\def\p{\mu}
\def\tp{{\tilde\mu}}
\def\tof{\xrightarrow{\rm flag}}
\def\toC{\xrightarrow{\rm Cay}}
\def\smod{{\sigma_{\rm mod}}}

\def\tits{\partial_{\rm Tits}}
\def\tmod{{\tau_{\rm mod}}}
\def\vi{\partial_{\rm I}}
\def\vii{\partial_{\rm II}}

\DeclareMathOperator{\cl}{cl}

\DeclareMathOperator{\isom}{Isom}
\DeclareMathOperator{\pr}{pr}
\DeclareMathOperator{\PSL}{PSL}
\DeclareMathOperator{\rl}{rl}
\DeclareMathOperator{\St}{St_{Fu}}
\def\geo{\partial_{\infty}}

\newcommand{\LT}[1]{\Lambda_{\tau_{\rm mod}}({#1})}

\newcommand{\coloneqq}{\mathrel{\mathop:}=}

\begin{document}
\maketitle
\begin{abstract}
The classical Klein--Maskit combination theorems provide sufficient conditions to construct new Kleinian groups using old ones.
There are two distinct but closely related combination theorems: The first deals with amalgamated free products, whereas the second deals with HNN extensions. 
This article gives analogs of both combination theorems for Anosov subgroups.
\end{abstract}

\setcounter{tocdepth}{1}
\tableofcontents

\section{Introduction}

In the theory of Kleinian groups (discrete isometry groups of $\HH^3$),
the {\em Klein--Maskit combination theorems} provide techniques to construct new Kleinian groups.
The history of combination theorems dates back to Klein's 1883 paper \cite{klein1883neue}, which gave sufficient conditions for a subgroup $\G$ of $G = \isom(\HH^3)$ generated by two discrete subgroups $\G_1$ and $\G_2$ of $G$ to be discrete and isomorphic to the free product $\G_1\star\G_2$.
Subsequently, Maskit \cite{maskit1965klein,maskit1968klein,maskit1971klein,maskit1993klein}, dealing with the cases of amalgamated free products and HNN extensions, gave far-reaching generalizations of the {\em Klein combination theorem}. Maskit's combination theorems also furnish sufficient conditions for the combined group to have certain geometrical properties, such as {\em convex-cocompactness} or {\em geometric-finiteness}.
See Maskit's book \cite{maskit:book} for an account of those results.
Later on, Ivascu \cite{Ivascu} and, more recently, Li-Ohshika-Wang \cite{LOW1,LOW2} extended the Klein--Maskit combination theorems to the setting of discrete isometry groups of higher dimensional hyperbolic spaces.
The combination theorems were further generalized in the context of group actions on Gromov-hyperbolic spaces; see \cite{baker2008combination,GITIK199965,pedroza2008combination,pedroza2009,MR2994828,MM,TW}.
We refer to \cite{MR4472055} for a recent survey of combination theorems.

In the recent years, {\em Anosov subgroups} of semisimple Lie groups have emerged as a higher-rank generalization of convex-cocompact Kleinian groups.
Our motivation for this article is to provide suitable analogs of the Klein--Maskit combination theorems in the context of Anosov subgroups.
See \Cref{mainthm:amalgam} and \Cref{mainthm:HNN} below, where we state our versions of the {\em first} and {\em second} Klein--Maskit combination theorems, respectively.
We note that the general strategy of our proof of these results is quite different from Maskit's proof and its subsequent generalization by Ivascu, and Li-Ohshika-Wang.
The main distinction comes from the fact that  Kleinian groups act as {\em discrete convergence groups} on the ideal boundary of the hyperbolic space. In contrast, such convergence property of general discrete subgroups is absent in the higher rank setting. 
However, the special class of discrete subgroups, called {\em regular subgroups}, which includes {\em Anosov subgroups}, exhibits a (more technical) form of convergence action on suitable partial flag varieties.
A significant part of the technical core of this paper is devoted to developing machinery to verify that our combined group retains the higher rank convergence property.
After that, the main difficult step is to construct an equivariant {\em boundary map} and demonstrate its continuity.  Furthermore, in contrast to our arguments, Maskit's proof and its subsequent generalizations for Kleinian groups relied upon proving that certain limit points are conical and we currently do not know how to do this for higher rank {\em regular antipodal} subgroups without constructing boundary maps in the, more special, Anosov setting.  

Previously, in our paper with Bernhard Leeb \cite{MR4002289}, we gave an earlier form of our Combination Theorem for Anosov subgroups; in that paper, using the {\em local-to-global principle} for {\em Morse quasigeodesics}, we proved a version of the Klein Combination Theorem for Anosov subgroups.
Moreover, we conjectured a {\em sharper} form of the Combination Theorem in that paper, which was recently confirmed in \cite{DK22}. The discussions in \cite{DK22,MR4002289} were limited only to the case of free products. 
The purpose of the present article is to deal with the case of {\em amalgams} (amalgamated free products and HNN extensions).
Nevertheless, our work partly builds on our earlier work in \cite{DK22}.
The generalization to amalgamated free products and HNN extensions presented in this paper requires overcoming many challenges not present in the case of free products. 

\subsection{Main results}  
Let $G$ be a real semisimple Lie group of noncompact type and with a finite center. We will assume some mild conditions on $G$ (see \Cref{rem:assumptionG}).
Let $X = G/K$ be the associated symmetric space, where $K$ is a maximal compact subgroup of $G$.
Let $\smod$ be a model spherical chamber in the Tits building $\tits X$ of $X$ and let $\iota : \smod \to \smod$ be the opposition involution.
We consider the class of {\em $\tmod$-Anosov subgroups} of $G$, where $\tmod$ is an $\iota$-invariant face of $\smod$.
For a discrete subgroup $\G$ of $G$, the {\em $\tmod$-limit set} of $\G$ in $\Ft$, the partial flag manifold associated to the face $\tmod$, is denoted by $\LT{\G}$. 
See \S\ref{sec:prelim} for the definitions.

Given discrete subgroups $\Gamma_A, \Gamma_B$ of $G$, a pair of compact subsets $A, B \subset \Ft$ with nonempty interiors is called an {\em interactive pair} for the triple $(\Gamma_A, \Gamma_B;\, \H)$, where $ \H = \Gamma_A \cap \Gamma_B$, if the interiors $A^\circ$ and $B^\circ$ of $A$ and $B$ are disjoint, and $\H$ leaves the sets $A$ and $B$ {\em precisely invariant}. This means that for all $\eta \in \H$, $\eta A = A$ and $\eta B = B$, and for all $\alpha \in \Gamma_A \setminus \H$ and $\beta \in \Gamma_B \setminus \H$, we have $\alpha B \subset A^\circ$ and $\beta A \subset B^\circ$.
We borrow this concept from Maskit's work. Cf. \Cref{pic2} for an illustration.

The first main result of this paper, which provides an analog of the first Klein--Maskit combination theorem \cite[Theorem VII.C.2]{maskit:book}, is as follows:

\begin{maintheorem}[Amalgamated free products]\label{mainthm:amalgam}
  Let $\G_A$ and $\G_B$ be $\tmod$-Anosov subgroups of $G$. Assume that $\H \coloneqq \G_A\cap \G_B$ is quasiconvex (see \Cref{def:quasiconvex}) in $\G_A$ or in $\G_B$.
 Suppose that there exists an {\em interactive pair} $(A,B)$  in $\Ft$ for $(\G_A, \G_B;\, \H)$ 
 such that the following conditions are satisfied:
\begin{enumerate}[(i)]
 \item The interiors of $A$ and $B$ are {\em antipodal} to each other (see \Cref{def:antipodality}).
 \item The pairs of subsets $(A,\,\LT{\G_B} \setminus \LT{\H})$ and $(B,\,\LT{\G_A} \setminus \LT{\H})$ of $\Ft$ are antipodal to each other.
\end{enumerate}
Then, the subgroup $\G = \< \G_A, \G_B\>$ of $G$ is $\tmod$-Anosov and is naturally isomorphic to the abstract amalgamated free product $\G_A \star_\H \G_B$.
\end{maintheorem}

\begin{figure}[h]
\centering
\begin{overpic}[scale=1,tics=8]{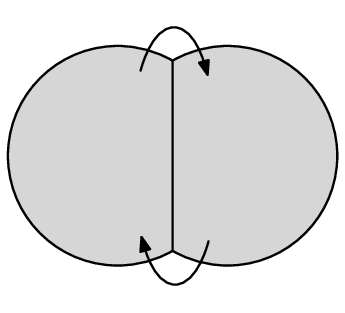}
\put(22,42){\huge $A$}
\put(66,42){\huge $B$}
\put(40,85){\footnotesize $(\G_B \setminus \H)$}
\put(40,1.5){\footnotesize $(\G_A \setminus \H)$}
\end{overpic}
\caption{An interactive pair.}
\label{pic2}
\end{figure}
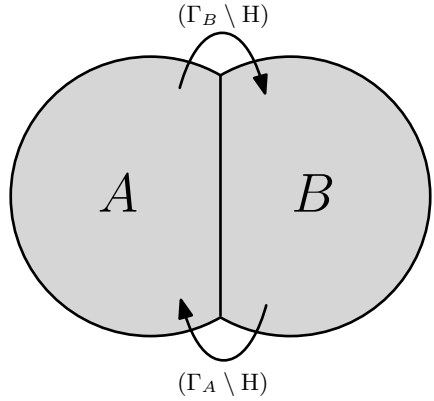

Given a discrete subgroup  $\P$ of $G$ and an element $f\in G$, 
a triple $(A, B_\pm)$ of compact subsets  of $\Ft$ with nonempty interiors
is called an {\em interactive triple} for the quadruple $(\P;\,\H_\pm;\,f)$, where 
\begin{equation}\label{def:Hpm}
\H_+ \coloneqq \P\cap ({f\P f^{-1}}) \quad\text{and}\quad
\H_- \coloneqq ({f^{-1}\P f}) \cap \P,
\end{equation}
if the interiors $A^\circ$, $B_-^\circ$, and $B_+^\circ$ of $A$, $B_-$, and $B_+$, respectively, are nonempty and pairwise disjoint, $B_- \cap B_+ = \emptyset$, $f^{\pm 1}(A) \subset B_\pm$, $f^{\pm 1}(B_\pm) \subset B_\pm^\circ$, and $\H_\pm$ leaves $B_\pm$ {\em precisely invariant}. This means $\H_\pm B_\pm = B_\pm$ and $\mu(B_\pm) \subset A^\circ$ for all $\mu \in \P\setminus \H_\pm$. Cf. \Cref{pic3}.

Our second main result, stated below, gives an analog of the second Klein--Maskit combination theorem \cite[Theorem VII.E.5]{maskit:book}.

\begin{maintheorem}[HNN extensions]\label{mainthm:HNN}
 Let $\P$ be a $\tmod$-Anosov subgroup of $G$, let $f\in G$, and let $\H_\pm$ be as in \eqref{def:Hpm}.
 Assume that $\H_+$ or $\H_-$ is quasiconvex in $\P$.
 Suppose that there exists an 
 {\em interactive triple} $(A,B_\pm)$ in $\Ft$ for $(\P; \,\H_\pm; \, f)$ 
 such that the following conditions are satisfied:
 \begin{enumerate}[(i)]
  \item The pairs of subsets $(A^\o,B_\pm^\o)$, $(B_+,B_-)$ of $\Ft$ are antipodal to each other. 
  \item $\LT{\P} \setminus \LT{\H_\pm}$ is antipodal to $B_\pm$.
 \end{enumerate}
 Then, the subgroup $\G = \< \P,f\>$ of $G$ is $\tmod$-Anosov and is naturally isomorphic to the abstract HNN extension $\Pf$ of $\P$ by $\phi$, where the isomorphism $\phi :\H_- \to{\H_+}$ is given by $\phi(\eta) = f\eta f^{-1}$, for all $\eta\in\H_-$.
\end{maintheorem}

\begin{figure}[h]
\centering
\begin{overpic}[scale=1,tics=8]{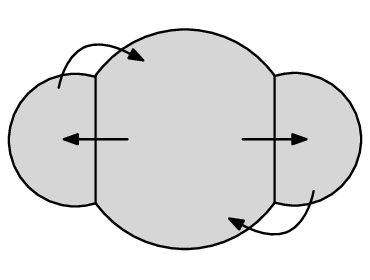}
\put(44,36){\huge $A$}
\put(85,36){\large$B_+$}
\put(7,36){\large$B_-$}
\put(27,41) {\small $f^{-1}$}
\put(67,41) {\small $f$}
\put(68,6){\footnotesize $(\P \setminus \H_+)$}
\put(10,67){\footnotesize $(\P \setminus \H_-)$}
\end{overpic}
\caption{An interactive triple.}
\label{pic3}
\end{figure}
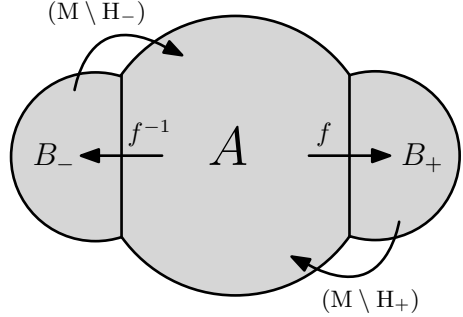

See \S\ref{sec:CT_A}, resp. \S\ref{sec:CT_B}, for the proof of \Cref{mainthm:amalgam}, resp. \Cref{mainthm:HNN}.

\begin{remark}\label{rem:relaxation}
 In the statement of \Cref{mainthm:amalgam}, one may relax the second condition (ii) if it is known that $\LT{\G_A}\cap \partial A = \LT{\H}=\LT{\G_B}\cap \partial B$.
 Similarly, in the statement of \Cref{mainthm:HNN}, one may relax the second condition (ii) if it is known that $\LT{\P}\cap \partial A = \LT{\H_+}\cup\LT{\H_-}$. 
 See \Cref{rem:relaxation_amalgam,rem:relaxation_HNN}.
\end{remark}

\begin{remark}
The ``converse'' of \Cref{mainthm:HNN} (and  \Cref{mainthm:amalgam}) does not hold in general: More precisely,  let $\P$ be a $\tmod$-Anosov subgroup of $G$ and $f \in G$. Suppose $\H_+ \coloneqq \P \cap (f \P f^{-1})$ or $\H_- \coloneqq (f^{-1} \P f) \cap \P$ is quasiconvex in $\P$. If the subgroup $\Gamma = \langle \P, f \rangle$ of $G$ is $\tmod$-Anosov and is naturally isomorphic to $\P *_\phi$, where $\phi: \H_- \to \H_+$ is an isomorphism given by $\phi(\eta) = f \eta f^{-1}$, $\eta \in \H_-$, it does not necessarily imply the existence of an interactive triple $(A, B_\pm)$ as stated in \Cref{mainthm:HNN}.
Here we discuss an example of such a failure when $G = \PSL(2,\C)$.
\end{remark}

\begin{example} Start with a compact hyperbolic 3-manifold  $M$ with a single totally geodesic boundary component of genus $2$ and let $M_1, M_2$ be two copies of $M$. Pick essential separating annuli $A_i\subset \partial M_i$ for $i=1, 2$, and glue $M_1, M_2$ along $A_1, A_2$.  The result is a compact 3-manifold $N$ with two homeomorphic boundary components $S_1, S_2$, both of genus $2$. The manifold $N$ has an essential properly embedded annulus $A$ (corresponding to $A_1, A_2$) with the boundary circles $C_i \subset S_i$, $i=1,2$. The manifold $N$ is irreducible and contains no essential annuli other than $A$. By Thurston's Hyperbolization Theorem (see \cite{MR2553578}, \cite{MR1677888} or \cite[Theorem A']{morgan84}), $N$ admits a convex-cocompact hyperbolic structure. Glue the two boundary surfaces of $N$ by a homeomorphism $\phi: S_1\to S_2$ such that  $\phi(C_1)$ is not homotopic  to $C_2$ in $S_2$. Then, the resulting closed manifold $K$ is irreducible, Haken, and atoroidal, and hence admits a hyperbolic structure, again by Thurston's theorem. The surfaces $S_1, S_2$ are incompressible and $\pi_1(S_1)$, $\pi_1(S_2)$ are quasiconvex in $\pi_1(K)$.

Although $\pi_1(K) = \pi_1(N)\star_{\phi_*}$ is a cocompact (hence Anosov) subgroup of $G=\PSL(2,\C)$, the quadruple $(\pi_1(N); \, \pi_1(S_1),\pi_1(S_2);\, f)$, where $f\in G$ corresponds to $\phi$, does not  admit an interactive triple in $\geo \mathbb{H}^3$: To see this, note that $C_1\subset S_1$ is homotopic $C_2 \subset S_2$ in $N$, and so $\H_+=\pi_1(S_1)$ and $\H_-=\pi_1(S_2)$  violate \Cref{cor:malnormalHNN}, a conclusion arising from the existence of an interactive triple.
\end{example}

Before concluding the introduction, let us illustrate the hypotheses and conclusions of \Cref{mainthm:amalgam} and \Cref{mainthm:HNN} in the following examples:

\subsection{Illustrative examples}\label{example}
Let $S$ be a closed, orientable surface of genus $g\ge 2$ and let $\G \coloneqq \pi_1(S)$.
Labourie's \cite{MR2221137} pioneering work showed that {\em Hitchin representations} $\rho: \G \to \PSL(d,\R)$ form a rich class of $\smod$-Anosov representations.
Let us denote by $\xi_\rho: \geo\G \to \Fs$ the associated $\G$-equivariant boundary map onto the limit set $\Lambda_{\rho(\G)} \subset \Fs$.
According to Fock-Goncharov \cite{MR2233852}, an important characteristic of the Hitchin representations is a certain {\em positivity property} of the limit map $\xi_\rho$; cf. \S\ref{sec:positive}.

\medskip
Given an essential {\em separating} simple closed curve $s \subset S$, let $[\eta]$ denote the conjugacy class in $\G$ representing $s$.
The curve $s$ cuts the surface into two subsurfaces, $S_A$ and $S_B$.
The group $\G$ can be written as an amalgamated free product
\[
 \G = \G_A \star_{\H} \G_B,
\]
where $\G_A = \pi_1(S_A)$, $\G_B = \pi_1(S_B)$, $\H = \< \eta \>$, for some $\eta\in [\eta]$, equipped with natural monomorphisms $\phi_A : \H \to \G_A$ and $\phi_B : \H \to \G_B$ induced by the inclusion homeomorphisms $s\hookrightarrow \partial S_A$ and $s\hookrightarrow \partial S_B$.
The image of $\H = \<\eta\>$ under a Hitchin representation $\rho : \G \to \PSL(d,\R)$ is a cyclic $\smod$-Anosov subgroup of $\PSL(d,\R)$. Let $\sigma_\pm\in\Lambda_{\rho(\G) }= \xi_\rho (\geo\G)$ denote the attractive/repulsive fixed points of $\rho(\eta)$ in $\Fs$.
The 2-point set $\{\sigma_+, \sigma_-\}$ cuts $\Lambda_{\rho(\G)}$ in a pair of closed arcs $c_A$ and $c_B$; we chose the names in such a way that $\Lambda_{\rho(\G_A)} \subset c_A$ and $\Lambda_{\rho(\G_B)} \subset c_B$.
See the left side of \Cref{pic1}.

Let $\Omega$ be the set of all points in $\Fs$ antipodal to both $\sigma_\pm$; this is the intersection of a pair of opposite maximal Schubert cells in $\Fs$:
\[
 \Omega = C(\sigma_+) \cap C(\sigma_-).
\]
There are two distinguished  connected components of $\Omega$, denoted by $A^\o$ and $B^\o$, whose closures $A$ and $B$, respectively, contain $c_A$ and $c_B$.
Both $A$ and $B$ are preserved by $\H$ because $\H$ preserves $c_A$ and $c_B$.
Using the positivity property of $\xi$ and the fact that for all $\alpha\in \rho(\G_A \setminus \H)$ and $\beta\in \rho(\G_B \setminus \H)$, 
$\a c_B \subset c_A \setminus \{ \sigma_\pm\}$ and $\beta c_A \subset c_B\setminus \{ \sigma_\pm\}$, it follows that
\[
 \a B \subset A^\o
 \quad\text{and}\quad
 \b A \subset B^\o,
\]
see \cite[Corollary 3.9]{Guichard:2021aa}.
Moreover, the interiors $A^\o$ and $B^\o$ are antipodal to each other, see \cite[Proposition 2.5(3)]{Guichard:2021aa}.
Therefore, $(A,B)$ is an interactive pair for $(\rho(\G_A),\rho(\G_B);\,\rho(\H))$, which satisfies the hypothesis of \Cref{mainthm:amalgam}.

Let us describe a continuous family of Hitchin representations $\rho_t: \G \to \PSL(d,\R)$  obtained by {\em bending} the given Hitchin representation  $\rho$ along $s$, a type of deformation that generalizes the classical Dehn twists along simple closed geodesics in a hyperbolic surface.
This family is parametrized by $t\in Z_0(\eta) $, where $Z_0(\eta) \cong \R^{d-1}$ is the identity component of the centralizer of $\eta$ in $\PSL(d,\R)$, and the $\smod$-Anosov property of this family can be verified by a simple application of \Cref{mainthm:amalgam}: 
Each connected component of $\Omega$ is also preserved by $Z_0(\eta)$. In particular, $Z_0(\eta)$ preserves $A$ and $B$.
For $t\in Z^0(\eta)$, let
\[
 \rho_t : \G_B \to \PSL(d,\R), \quad
 \rho_t(\beta) = t \rho(\b) t^{-1}.
\]
Clearly, $\rho_t\vert_\H = \rho\vert_\H$.
Moreover, we observe that $(A,B)$ is still an interactive pair for the triple $(\rho(\G_A),\rho_t(\G_B);\,\rho(\H))$.
Therefore,  \Cref{mainthm:amalgam} directly verifies that 
\[
 \rho_t : \G = \G_A \star_\H \G_B \to \<\rho(\G_A),{t\rho(\G_B)t^{-1}}\> <  \PSL(d,\R),
\]
is injective, and its image is a $\smod$-Anosov subgroup of $G$.

\medskip
Similarly, if $s$ is an essential {\em non-separating} simple closed curve in $S$,
then the fundamental group $\G = \pi_1(S)$ can be written as an HNN extension
\[
 \G = \Pf, \quad \phi: \H_- \xrightarrow{\cong} \H_+,
\]
where $\P$ is the fundamental group of the surface $S'$ obtained by cutting $S$ along $s$, $\H_- = \<\eta\>$ and $\H_+ = \<\eta'\>$ are the images of the homomorphisms of the fundamental groups induced by the two inclusion homeomorphisms $s\hookrightarrow \partial_- S'$ and $s \hookrightarrow \partial_+ S'$, and $\phi: \H_- \to \H_+$ is the isomorphism induced by the conjugation by some element $f\in \G$. 

Let $\rho: \G\to \PSL(d,\R)$ be a Hitchin representation with associated $\G$-equivariant limit map $\xi_\rho: \geo\G \to \Fs$.
The attractive/repulsive fixed point sets $\{\sigma_+, \sigma_-\}$ and $\{\sigma'_+, \sigma'_-\}$ of $\eta$ and $\eta'$, respectively, cut the limit set $\Lambda_{\rho(\G)} = \xi_\rho(\geo\G)$ in four closed arcs $c_{A_+}$, $c_{A_-}$, $c_{B_+}$, and $c_{B_-}$ (see the right side of \Cref{pic1}).
We have $\rho(f)\{\sigma_\pm\} = \{\sigma_\pm'\}$, $\rho(f)(c_{A_+}\cup c_{A_-} \cup c_{{B_+}})\subset c_{{B_+}}$, and $\rho(f)^{-1}(c_{A_+}\cup c_{A_-} \cup c_{B_-})\subset c_{B_-}$.

\begin{figure}[h]
\centering
\begin{overpic}[scale=.8,tics=5]{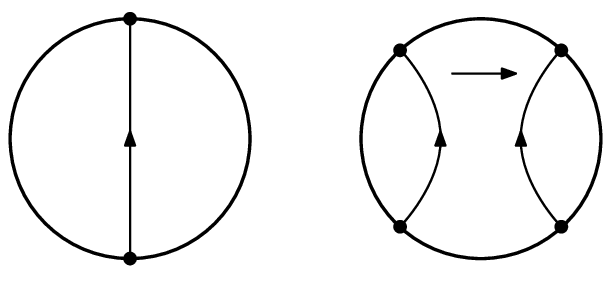}
\put(42,24){$c_{B}$}
\put(-4,24){$c_{A}$}
\put(20,45.5){\small$\sigma_+$}
\put(20,1){\small$\sigma_-$}
\put(17,24){$\eta$}
\put(77,45.5){$c_{A_+}$}
\put(77,1){$c_{A_-}$}
\put(51,24){$c_{B_-}$}
\put(100,24){$c_{B_+}$}
\put(60,40){\small$\sigma_+$}
\put(93.5,40){\small$\sigma'_+$}
\put(60,7){\small$\sigma_-$}
\put(93.5,7){\small$\sigma'_-$}
\put(77,36.4){$f$}
\put(68,24){$\eta$}
\put(87,24){$\eta'$}
\end{overpic}
\caption{}
\label{pic1}
\end{figure}

Define $B_-,{B_+} \subset \Fs$ to be the closure of connected component of $C(\sigma_+)\cap C(\sigma_-)$ containing the arc $c_{B_-}$ and $c_{{B_+}}$, respectively.
Lastly, define $A$ to be the union of $A_+$ and $A_-$, where $A_\pm$ are  the closures of the  connected component of $C(\sigma_\pm)\cap C(\sigma_\pm)$ containing the arcs $c_{A_\pm}$, respectively.
 Then, using the positivity of  $\xi_\rho$, one can show that $(A,B_\pm)$ is an interactive triple for $(\rho(\P);\, \rho(\H_\pm);\, \rho(f))$ (compare with the amalgamated free product case discussed above), thus verifying  the hypothesis of \Cref{mainthm:HNN}.  

As before, we obtain a continuous family of Hitchin representations $\rho_t: \G \to \PSL(d,\R)$, parametrized by $t\in Z_0(\eta)$, by bending $\rho$ along $s$:
Observe that for any $t\in Z_0(\eta)$, $(A,{B_\pm})$ is an interactive triple for $(\rho(\P);\, \rho({\H_\pm});\, \rho(f)\cdot t)$.
Therefore, \Cref{mainthm:HNN} directly verifies  that \[\rho_t : \G = \Pf \to \<\rho({\P}), \rho(f)\cdot t\> < \PSL(d,\R)\] is injective with a $\smod$-Anosov image.

\subsection*{Organization of the paper}
In \S\ref{sec:prelim_grp}, we present some preliminary material and results on amalgamated free products and HNN extensions of hyperbolic groups. This section presents some results (e.g., \Cref{lem:proj}, \Cref{thm:alternating_one,prop:reln}) crucial in the  proof of our main theorems.
Then, in \S\ref{sec:prelim} we review some background material on discrete groups acting on symmetric spaces of noncompact type and present some lemmas (see \S\ref{sec:regularsequences}), which are also frequently used in the  proof of our main results to verify regularity and flag-convergence of certain sequences.
In \S\ref{sec:CT_A}, resp. \S\ref{sec:CT_B}, we prove \Cref{mainthm:amalgam}, resp. \Cref{mainthm:HNN}.
Finally, in \S\ref{sec:positive}, using the Combination Theorems, we give a proof of a recent result of Guichard-Labourie-Wienhard \cite{Guichard:2021aa} stating that $\Theta$-positive representations are $\Theta$-Anosov.

\subsection*{Acknowledgements}
We are grateful to our referees for their valuable feedback.
S.~D. completed this work at Yale University and extends his gratitude to the mathematics department for providing an exceptional working environment.

\section{Preliminaries on amalgams of hyperbolic groups}\label{sec:prelim_grp}

\subsection{Amalgamated free products}\label{sec:reduced}

Let $\G_A$, $\G_B$, and $\H$ be abstract groups together with monomorphisms $\phi_A : \H \to \G_A$, $\phi_B : \H \to \G_B$.
 The {\em free product of $\G_A$ and $\G_B$ amalgamated along $\H$}, denoted by $\G_A \star_{\H} \G_B$, has a presentation
 \[
 \< S_A,\, S_B \mid R_A,\, R_B,\, \phi_A (\eta)\phi_B(\eta)^{-1}, \eta \in \H \>,
 \]
 where $\< S_A \mid R_A\>$ and $\< S_B \mid R_B\>$ are presentations of $\G_A$ and $\G_B$, respectively.
 We will identify $\H$ with $\phi_A(\H)<\G_A$ and $\phi_B(\H)<\G_B$ via the monomorphisms $\phi_A$ and $\phi_B$, respectively.

Let $\G = \G_A \star_{\H}\G_B$.
A {\em normal form} of an element $\g\in \G$ is an expression
\begin{equation}\label{eqn:reduced}
 \g = \g_1\g_2\cdots \g_l,
\end{equation}
such that the following conditions are satisfied:
\begin{enumerate}[(i)]
 \item 
 Each $\g_i$ lies either in $\G_A$, or in $\G_B$.
 Moreover, if $l\ge 2$, then none of the letters $\g_i$ belong to $\H$.

 \item No two successive letters $\g_i,\g_{i+1}$ above lie in the same group $\G_A$ or $\G_B$.
\end{enumerate}

Unless $\H$ is trivial, the normal form of $\g$ given by \eqref{eqn:reduced} is not necessarily unique.
However, any other normal form of $\g$ is obtained by a sequence of finite moves consisting of replacing a consecutive pair $\g_i \g_{i+1}$ in \eqref{eqn:reduced} by $(\g_i\eta_i)(\eta_i^{-1}\g_{i+1})$, where $\eta_i\in \H$.
This is a consequence of the Normal Form Theorem; see \cite[Theorem IV.2.6]{Lyndon-Schupp}.
It follows that any two normal forms of $\g$ have the same number of letters.
 For $\g\in\G\setminus \H$, the {\em relative length} of $\g$, denoted by $\rl(\g)$, is 
 the number of letters in a(ny) normal form of $\g$.
 If $\g\in\H$, then define $\rl(\g) = 0$.

\begin{definition}[Alternating sequences]\label{def:alternating}
 A sequence $(\omega_k)$ in $\G$ is called {\em type A alternating} if there exists a pair of sequences, $(\a_n)$ in $\G_A \setminus \H$ and $(\b_n)$ in $\G_B\setminus \H$, such that
\begin{equation}\label{eqn:alternating_A}
   \omega_k =
\begin{cases}
 \a_1\b_1\cdots \b_{m-1} \a_{m}, & k=2m-1,\\
 \a_1\b_1\cdots \a_m\b_m, & k=2m.
\end{cases}
\end{equation}
 Similarly, a sequence $(\omega_k)$ in $\G$ is called {\em type B alternating} if there exists a pair of sequences, $(\a_n)$ in $\G_A \setminus \H$ and $(\b_n)$ in $\G_B\setminus \H$, such that
 \[
  \omega_k =
\begin{cases}
 \b_1\a_1\cdots \a_{m-1} \b_{m}, & k=2m-1,\\
 \b_1\a_1\cdots \b_m\a_m, & k=2m.
\end{cases}
 \]
\end{definition}
 
\subsection{HNN extensions}
 Let $\P$ be an abstract group and $\H_\pm < \P$ be a pair of subgroups.
 Suppose that $\phi:\H_-\to\H_+$ is an isomorphism.
 The {\em HNN extension of $\P$ by $\phi$}, denoted by $\Pf$,
 has a presentation 
 \[
 \< S_\P,f \mid R_\P,  f\eta f^{-1}(\phi(\eta))^{-1}, \eta\in\H_-\>,
 \]
 where $\< S_\P\mid R_\P\>$ is a presentation of $\P$.
 
Every element $\g\in\G = \Pf$ can be  written  in the form
\begin{equation}\label{eqn:normal_form_HNN}
  \g = \p_{0} f^{\epsilon_{1}}\p_1 \cdots f^{\epsilon_n} \p_n,
\end{equation}
for some $n\ge 0$, where each $\p_i$ is in $\P$, and each $\epsilon_i$ is either $1$ or $-1$, such that the following conditions are satisfied:
\begin{enumerate}[(i)]
 \item For $i=1,\dots,n$, if $\epsilon_{i} =-1$ and $\p_i \in \H_+$, then $\epsilon_{i+1} = -1$.
 \item For $i=1,\dots,n$, if $\epsilon_{i} =1$ and $\p_i \in \H_-$, then $\epsilon_{i+1}=1$.
\end{enumerate}
Such an expression \eqref{eqn:normal_form_HNN} is called a {\em normal form} of $\g$.

Although the normal form of an element is not unique, {\em Britton's Lemma} (see \cite[p.181]{Lyndon-Schupp}) establishes certain uniqueness properties 
of the decomposition \eqref{eqn:normal_form_HNN} of $\g$. In particular, the {\em relative length} $\rl(\g)$ of $\g$,
i.e., the total number of letters $f$ and $f^{-1}$ appearing in \eqref{eqn:normal_form_HNN},
is unique.

\begin{definition}[Alternating sequences]\label{defn:alt_HNN}
 A sequence $(\omega_n)$ in $\Pf$ is called {\em alternating} if
 there exist  sequences $(\p_n)$ in $\P$ and $(\epsilon_n)$ in $\{\pm1\}$, satisfying
\begin{equation}\label{eqn:alt_HNNsign}
 \epsilon_n=-1, \p_n\in \H_+ \implies \epsilon_{n+1}=-1,
 \quad\text{and}\quad
 \epsilon_n=1, \p_n\in \H_- \implies \epsilon_{n+1}=1
\end{equation}
for all $n\in\N$,
and some element $\p_0\in\P$ such that
\begin{equation}\label{eqn:alt_HNN}
   \omega_n = \p_0 f^{\epsilon_1} \p_1 f^{\epsilon_2}\p_2 \cdots f^{\epsilon_{n-1}}\p_{n-1}f^{\epsilon_n}.
\end{equation}
\end{definition}
 
\begin{remark}\label{rem:altstring}
The condition \eqref{eqn:alt_HNNsign} simply implies that the word in \eqref{eqn:alt_HNN}
is a normal form.
It is also useful to think about an alternating sequence $(\omega_n)$ in $\Pf$ as an infinite string of letters:
\[
 \p_0, f^{\epsilon_1}, \p_1, f^{\epsilon_2}, \p_2, f^{\epsilon_3}, \p_3 \dots,
\]
where $\p_i$ and $\epsilon_i$ satisfy the condition \eqref{eqn:alt_HNNsign}.
\end{remark}

\subsection{Bass-Serre trees}\label{sec:BStree}
  
In this paper, we will be using the formalism of {\em Bass-Serre trees} $T$ associated with amalgamated free products and HNN extensions. A detailed treatment of this material can be found in Serre's book  \cite{trees}; see also \cite{Drutu-Kapovich} for a more topological viewpoint on the construction.

\medskip
(i) Consider an amalgamated free product $\G= \G_A\star_{\H} \G_B$. The vertex set $V(T)$ of the tree $T$ is the set of cosets $\g\G_A, \g\G_B$, $\g\in \G$. Accordingly, the vertex set $V(T)$ is bicolored, with one color corresponding to the cosets $\g\G_A$ and the other color corresponding to the cosets $\g\G_B$. 
The group $\G$ acts on $V(T)$ by the left multiplication.

Edges of $T$ are defined so that $T$ is a bipartite graph, i.e., vertices of the same color are never connected by an edge. The cosets $\g_1\G_A, \g_2\G_B $ are connected by  an edge whenever there exists $\alpha\in \G_A$ such that 
$$
\g_1\G_B =\g_2\alpha\G_B, 
$$ 
 equivalently, there exists $\beta\in \G_B$ such that
 $$
 \g_1\beta\G_A =\g_2\G_A. 
 $$ 
 For instance, the vertices represented by the cosets $\G_A$, $\G_B$ are connected by an edge $e\in E(T)$ since there exists an element $\eta\in \H< \G_A$ such that $\G_B=\eta\G_B $. Moreover, all elements $\alpha\in \G_A$ such that 
 $\alpha\G_B=\G_B $ necessarily belong to $\H$. We, thus, label the edge $e$ by the coset $1_\G\cdot \H=\H$. 
 
 The left multiplication by the elements of $\Gamma$ preserves the incidence relation between the vertices of $T$. Accordingly, the edges of $T$ are labeled by the cosets $\g\H $, $\g\in G$. 
The $\G$-stabilizer of the vertex $\g\G_A $ equals $\g\G_A \g^{-1}$, while the  
$\G$-stabilizer of the vertex $\g\G_B $ equals $\g\G_B \g^{-1}$ and the $\G$-stabilizer of an edge labeled $\g\H$ equals 
$\g\H \g^{-1}$. 

\medskip

(ii) Consider an HNN extension $\G = \P\star_{\phi:\H_-\to\H_+}$. The vertex set of $T$ consists of left cosets of $\P$ in $\G$.
The edge set of $T$ consists of edges of two types: The left $\H_\pm$-cosets in $\G$. The edge 
$\g\H_+$ connects $\g\P$ to $\g f\P$, while the edge $\g\H_-$ connects $\g\P$ and $\g f^{-1}\P$.  The 
$\G$-action on $T$ is defined by: $\g: \g'\P\mapsto  \g\g'\P$. The $\G$-action is transitive on the set of vertices and edges of $T$. 

\subsection{Hyperbolic groups}
Let $\G$ be a finitely-generated group equipped with a left-invariant word metric, denoted by $d_\G$.
We use the notation $|\cdot|$ to denote the word length of elements of $\G$, i.e., $|\g| = d_\G(1_\G,\g)$.
Recall that $\G$ is called {\em (word) hyperbolic} if there exists $\delta\ge0$ such that
$(\G,d_\G)$ is a {\em $\delta$-hyperbolic} metric space: That is,
for all $f,g,h,w\in\G$,
\[
 (f,g)_w \ge \min\{(f,h)_w,(g,h)_w\} - \delta,
\]
where $(f,g)_w$ denotes the {\em Gromov product} of $f$ and $g$ with respect to $w$:
\[
 (f,g)_w = \frac{1}{2}(d_\G(f,w) + d_\G(g,w) - d_\G(f,g)).
\]
It is a basic fact that the property of being hyperbolic does not depend on the choice of a word metric $d_\G$, but the constant $\delta$ possibly depends on the chosen metric $d_\G$.

Let $\G$ be a hyperbolic group equipped with a word metric $d_\G$.
A {\em (discrete) geodesic} in $\G$ is an isometric embedding $c: I \to \G$ of an interval $I\subset \Z$.
Such a geodesic $c: I \to \G$ is called a {\em segment}, {\em ray}, or {\em line} when $I$ is bounded, $I$ is only bounded below, or $I = \Z$, respectively.

\medskip
The {\em Gromov boundary}  of $\G$, denoted by $\geo\G$, is the set of equivalence classes of asymptotic geodesic rays in $\G$, which gives a natural compactification of $(\G, d_\G)$,
\[
 \bG = \G \sqcup \geo\G.
\]
The topology of $\bG$ is understood as follows:
A pair of sequences $(\g_n)$ and $(\g_n')$ in $\G$ are said to {\em fellow travel} if 
\[
 (\g_n,\g_n')_{1_\G} \to \infty,
 \quad\text{as } n\to\infty.
\] 
A sequence $(\g_n)$ in $\G$ {\em converges} to a point $\e\in\geo\G$, which is denoted by
\[
 \g_n \toC \e,
\] 
if  $(\g_n)$ fellow-travels the image sequence $(c(n))$ of a(ny) geodesic ray $c:\N\to\G$ asymptotic to $\e$.

The following result can be checked easily using the  definitions above:

\begin{lemma}\label{lem:ft}
 Fellow traveling sequences in $\G$ have the same accumulation set in $\geo\G$.
\end{lemma}

\subsection{Nearest point projections to quasiconvex subgroups}
Let $\G$ be a hyperbolic group equipped with a left-invariant word metric $d_\G$.

\begin{definition}[Quasiconvexity]\ \label{def:quasiconvex}
\begin{enumerate}
\item A subset $Y\subset \G$ is called {\em quasiconvex} 
if there exists $K\ge 0$ such that for all $y_1, y_2 \in Y$, any geodesic segment in $\G$ connecting $y_1$ and $y_2$ lies in the $K$-neighborhood of $Y$.

\item A subgroup $\H <\G$ is called a {\em quasiconvex} if $\H$, as a subset of $\G$, is quasiconvex.
\end{enumerate}
\end{definition}

For a quasiconvex subset $Y\subset \G$, we choose
a nearest point projection map
\[
 \pr_Y : \G \to Y.
\]
Note that nearest point projections are not necessarily unique, but since $Y$ is quasiconvex, any two such maps are a finite distance apart.

\begin{lemma}\label{lem:proj_fellow_travel}
Let $Y\subset \G$ be a quasiconvex subset.
 For every $\g\in\G$, $\pr_Y(\g) \in Y$
 lies in a uniform neighborhood
 of any geodesic segment in $\G$ connecting $\g$ to $Y$. 
\end{lemma}

\begin{proof}
 Let $K\ge 0$ be a quasiconvexity constant for $Y$ and $\delta\ge 0$ be a hyperbolicity constant for $(\G,d_\G)$.
 Let $[y_0,\g]$ be a geodesic segment in $\G$ from $y_0\in Y$ to $\g$.
 
\begin{figure}[h]
\centering
\begin{overpic}[scale=.9,tics=8]{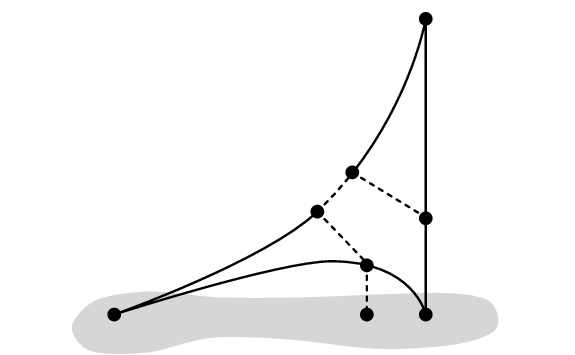}
\put(17,3){\small$y_0$}
\put(62,3){\small$y'$}
\put(65,17){\small$y$}
\put(73,3){\small$\g' = \pr_Y(\g)$}
\put(75.5,22.5){\small$x$}
\put(75.5,58){\small$\g$}
\put(55,11.3){\Small$K\ge$}
\put(52,18.5){\Small$2\delta\ge$}
\put(66.3,29.5){\Small$\le 2\delta$}
\put(52,28){\small$1=$}
\put(45,5){\large$Y$}
\end{overpic}
\caption{}
\label{pic5}
\end{figure}

 Consider a geodesic triangle in $\G$ with vertices $y_0$, $\g$, and $\g'\coloneqq \pr_Y(\g)$, whose edge connecting $y_0$ to $\g$ is $[y_0,\g]$.
 Since geodesic triangles in $\G$ are $2\delta$-thin, we have
 \[
  [y_0,\g] \subset N_{2\delta}([y_0,\g']\cup [\g,\g']),
 \]
 where $N_r(\cdot)$ denotes the $r$-neighborhood in $(\G,d_\G)$.
 In particular, there exist points $x\in [\g,\g']$ and $y\in [y_0,\g']$ such that
 $d_\G(x,[y_0,\g]) \le 2\delta$, $d_\G(y,[y_0,\g]) \le 2\delta$, and $d_\G(x,y) \le 4\delta + 1$.
 See \Cref{pic5} for an illustration of the points in the Cayley graph of $\G$.
 Since $[y_0,\g']$ lies in the $K$-neighborhood of $Y$, there exists $y'\in Y$, which is at most $K$ distance away from $y$.
 
Since $\g'$ is also a nearest point in $Y$ to $x$, $d_\G(x,\g') \le d_\G(x,y') \le K + 4\delta + 1$.
Therefore, $\g'$ is at most $K+6\delta+1$ distance away from $[y_0,\g]$.
\end{proof}

For a subset $Y\subset \G$, let $\geo Y \subset \geo\G$ denote the set of all accumulation points of $Y$ in $\bG = \G \sqcup \geo\G$.

\begin{corollary}\label{cor:proj_fellow_travel}
Let $Y \subset \G$ be a quasiconvex subset.
 A sequence $(\g_n)$ in $\G$ has an accumulation point
in $\geo Y$ if and only if $(\pr_Y(\g_n))$ is unbounded in $Y$.
\end{corollary}

\begin{proof}
 For the ``if'' part, suppose that $(\pr_Y(\g_n))$ is unbounded.
 After passing to a subsequence, we assume that $|\pr_Y(\g_n)| \to\infty$, as $n\to\infty$. Applying \Cref{lem:proj_fellow_travel}, it follows
 that $(\g_n)$ and $(\pr_Y(\g_n))$ fellow travel.
 Thus, by \Cref{lem:ft}, $(\g_n)$ has an accumulation point in $\geo Y$.
 
 For the ``only if'' part, we prove the contrapositive statement. Let us assume that $(\pr_Y(\g_n))$ is bounded. We show that $(\g_n)$ has no accumulation points in $\geo Y$.
 We argue by contradiction:
 Suppose that $\e\in\geo Y$ is an accumulation point of $(\g_n)$.
 After extraction, we may assume that $\g_n\toC \e$.
 Let $(y_n)$ be any sequence in $Y$ such that $y_n\toC \e$.
 Consequently, we must have $(\g_n,y_n)_{1_\G} \to \infty$, as $n\to\infty$.
 However, \Cref{lem:proj_fellow_travel} can be restated as
 \[
  \sup_{\g\in\G, y\in Y}(\g,y)_{\pr_Y(\g)} < \infty.
 \]
 Since $(\pr_Y(\g_n))$ is bounded, the above implies
 \[
  \sup_{m,n}(\g_n,y_m)_{1_\G} < \infty,
 \]
a contradiction.
\end{proof}

\begin{lemma}\label{lem:proj_one}
Let $\H< \G$ be a quasiconvex subgroup.
For any sequence $(\g_n)$ in $\G$,
consider the sequence $(\hg_n)$ given by $\hg_n = \pr_\H(\g_n)^{-1} \g_n$.
\begin{enumerate}[(i)]
 \item
 Regarded as a sequence in the compact space $\bG = \G \sqcup \geo\G$, $(\hg_n)$ has no accumulation points in $\geo \H$.

 \item
 Suppose that $(\g_n)$ diverges away from $\H$, i.e.,
 $
  d_\G(\H,\g_n) \to\infty,
 $
 as $n\to\infty$.
 Then, the accumulation set of $(\g_n^{-1})$ in $\geo\G$ is the same as that of $(\hg_n^{-1})$.
\end{enumerate}
\end{lemma}

\begin{proof}
We first observe that for all $\g\in\G$, the identity element is a nearest point in $\H$ to $\hg \coloneqq \pr_\H(\g)^{-1}\g$:
Indeed, for all $\eta\in\H$,
\begin{equation}\label{eqn:projection_ineq}
  d_\G (\eta ,\hg) = d_\G (\pr_\H(\g)\eta ,\g) \ge d_\G (\pr_\H(\g),\g),
\end{equation}
where the inequality follows from the fact that $\pr_\H(\g)$ is a nearest point in $\H$ to $\g$.
Moreover, $d_\G (1_\G ,\hg) = d_\G (\pr_\H(\g),\g)$. 
Hence, \eqref{eqn:projection_ineq} implies that $d_\G (\H ,\hg) \ge d_\G (1_\G ,\hg)$.

Therefore, for any sequence $(\g_n)$ in $\G$,
$
 \{ \pr_\H(\hg_n) \mid  n\in\N \} \subset \H
$
is bounded.
Part (i) now follows by applying \Cref{cor:proj_fellow_travel}.

\medskip
We prove part (ii):
Since $\pr_\H(\g_n)$ lies uniformly close to any geodesic in $\G$ connecting $1_\G$ to $\g_n$ (see \Cref{lem:proj_fellow_travel})
and $d_\G(\H,\g_n) \to\infty$, it follows that
\[
 |\g_n| - |\pr_\H(\g_n)| \to\infty,
\]
as $n\to\infty$.
Thus
\begin{align*}
 (\g_n^{-1},\hg_n^{-1})_{1_\G} &= 
 \frac{1}{2}(|\g_n^{-1}| + |\hg_n^{-1}| - d_\G(\g_n^{-1},\hg_n^{-1}))\\
 &=\frac{1}{2}(|\g_n| + |\hg_n| - |\g_n\hg_n^{-1}|)\\
 &=\frac{1}{2}(|\g_n| + |\hg_n| - |\pr_\H(\g_n)|) \to\infty,
\end{align*}
as $n\to\infty$.
In other words, $(\g_n^{-1})$ and $(\hg_n^{-1})$ fellow travel.
By \Cref{lem:ft}, $(\g_n^{-1})$ and $(\hg_n^{-1})$ have the same accumulation sets in $\geo\G$.
\end{proof}

An equivalent statement of the above result, which we will often use, is as follows:

\begin{lemma}\label{lem:proj}
Let $\H< \G$ be a quasiconvex subgroup.
For any sequence $(\g_n)$ in $\G$,
consider the sequence $(\hg_n)$ given by $\hg_n =  \g_n\pr_\H(\g_n^{-1})$.
\begin{enumerate}[(i)]
 \item
 $(\hg_n^{-1})$ has no accumulation points in $\geo \H$.

 \item
 If $(\g_n^{-1})$ diverges away from $\H$,
 then the accumulation sets of $(\g_n)$ and $(\hg_n)$ in $\geo\G$ coincide.
\end{enumerate}
\end{lemma}

\subsection{Amalgams of hyperbolic groups}
For the rest of this section,
we restrict our discussion to amalgams (amalgamated free products and HNN extensions) of hyperbolic groups. 
The  Bestvina-Feighn Combination Theorem, \cite{Bestvina-Feighn}, provides some sufficient conditions for the hyperbolicity of amalgams. 
We review this theorem in the weakly malnormal case (although their actual result is much more general):

\begin{definition}\label{def:malnormal}
A subgroup $\H$ of a group $\G$ is said to be  {\em weakly malnormal} if, for every $\g\in \G \setminus \H$, the subgroup 
$\g\H\g^{-1}\cap \H$ is finite.  
\end{definition} 

\begin{theorem}[Bestvina-Feighn, \cite{Bestvina-Feighn}]\label{thm:Bestvina-Feign}
Let $\G_A$, $\G_B$, and $\P$ be hyperbolic groups.
\begin{enumerate}[(i)]
 \item If $\H< \G_A$ and $\H< \G_B$ are quasiconvex, weakly malnormal subgroups, then $\G_A \star_\H \G_B$ is hyperbolic. 
 \item If $\H_\pm < \P$ are isomorphic (by $\phi:\H_-\to \H_+$), quasiconvex, weakly malnormal subgroups such that for all $\p\in\P$, $\H_-\cap \p\H_+\p^{-1}$ is finite, then $\Pf$ is hyperbolic.
\end{enumerate}
\end{theorem}

See also \cite[Theorems 1 \& 2]{MR1390041}.

This theorem has several addendums that we will use in what follows:

\begin{theorem}[Mitra, {\cite{Mitra1998}}]\label{thm:mj}
  Under the assumptions, the subgroups $\G_A$ and $\G_B$ (resp. $\P$) in \Cref{thm:Bestvina-Feign} are quasiconvex in $\G_A \star_\H \G_B$ (resp. $\Pf$). 
\end{theorem}

\subsection{Boundary of an amalgam} \label{sec:boundary}

Our next goal is to describe the Gromov boundaries of the  amalgamated free products $\G=\G_A\star_\H \G_B$ and 
HNN extensions $\G = \P\star_{\phi: \H_-\to \H_+}$ under certain extra assumptions. We will be assuming that the groups  $\G_A$, $\G_B$, and $\P$ are hyperbolic, $\H$ is weakly malnormal and quasiconvex in $\G_A, \G_B$ (in the amalgamated free product case) and,  $\H_\pm$ are weakly malnormal and quasiconvex in $\P$ and the intersection $\H_-\cap \p\H_+\p^{-1}$ is finite, for all $\p\in\P$ (in the HNN extension case). 
Under these assumptions, $\G$ is hyperbolic (see \Cref{thm:Bestvina-Feign}).
Moreover, in the case of HNN extensions, we let $f\in \G$ denote the {\em stable letter}, the element corresponding to the subgroup isomorphism $\phi: \H_-\to \H_+$:
$$
f \eta f^{-1}= \phi(\eta), \quad \eta\in \H_-. 
$$

Our description of the boundary follows \cite[7.3]{Kapovich-Sardar}, to which we refer the reader for details and proofs. We will describe the boundary (mainly) in the case of amalgamated free products since the HNN extension case is similar.

\medskip
Let $T$ denote the Bass-Serre tree associated with the  amalgamated free product $\G=\G_A\star_\H \G_B$ (or the HNN extension); see \S\ref{sec:BStree}. The group 
$\G$ acts on $T$ with vertex-stabilizers (the {\em vertex-subgroups} $\G_v$) that are conjugates of $\G_A, \G_B$ (or $\P$ in the HNN extension case) and edge-stabilizers ({\em edge-subgroups} $\G_e$) which are conjugates of $\H$.

Define a {\em tree of topological spaces} as follows:
To each  $v\in V(T), e\in E(T)$, we associate the Gromov boundary $\geo \G_v$, $\geo \G_e$. Whenever $v$ is a vertex of an edge $e$, we have the inclusion homomorphism $\G_e\to \G_v$, which induces a topological embedding 
$f_{ev}: \geo \G_e\to \geo \G_v$. This data yields a tree of topological spaces. The tree gives rise to a topological space 
$\vi\G$, 
the {\em topological realization} of the tree of topological spaces, 
by taking the push-out of the maps $f_{ev}$: The topological space $\vi\G$ is a union of Gromov boundaries  $\geo \G_v$, $v\in V(T)$. More precisely, it is the quotient of the disjoint union of these boundaries by the equivalence relation defined as follows. For every edge $e=[v,w]$ of $T$, we have $\xi\in \geo \G_v$ is equivalent to $\eta\in \geo \G_w$ whenever there exists $\zeta\in \geo \G_e$ such that $f_{ev}(\zeta)=\xi, f_{ew}(\zeta)=\eta$. The group $\G$ acts on $\vi\G$ via the projection of the natural $\G$-action on 
$$
\coprod_{v\in V(T)} \geo \G_v. 
$$
In particular,  $\vi\G$ is either the union of the $\G$-orbits of $\geo \G_A\cup \geo \G_B$ (in the amalgamated free product case) or $\geo \P$ (in the HNN extension case). 

The weak malnormality assumption on the amalgam implies that whenever $e\ne e'$ are distinct edges of $T$, $\geo \G_e\cap \geo \G_{e'}=\emptyset$ in $\vi\G$.
 Accordingly, whenever the distance between vertices $v, w$ is $>1$, 
\begin{equation}\label{boundary_disjoint}
 \geo \G_v\cap \geo \G_w=\emptyset. 
\end{equation}
Moreover, the weak malnormality assumption also implies that all vertex-subgroups $\G_v$ and edge-subgroups $\G_e$ are quasiconvex in $\G$. Hence, one obtains a natural $\G$-equivariant inclusion map  
$\vi\G\to \geo \G$. It turns out that this map is injective and continuous. However, in general, this map need not be a topological embedding. In what follows, we will identify $\vi\G$ with its image in $\geo \G$. 

The Gromov boundary $\geo\G$ of $\G$ is the disjoint union $\G$-invariant subsets
\[
 \geo\G = \vi\G \sqcup \vii\G,
\]
where $\vii\G \coloneqq \geo\G \setminus\vi\G$.
Elements of $\vi\G$, resp. $\vii\G$, are called {\em type I}, resp. {\em type II}, (ideal) boundary points of $\G$.

The second part of the boundary, $\vii\G$, of $\geo \G$ admits a $\G$-equivariant continuous bijection to $\geo T$. (For instance, 
$\geo \<f\>$ is the 2-point subset of $\vii\G$ corresponding to the fixed-point set of $f$ in $\geo T$.)  

\begin{proposition}\label{thm:alternating_one}
For every point $\e\in\vii \G$, there exists an alternating sequence $(\omega_n)$ in $\G$ converging (in $\bG$) to $\e$ such that the following holds:
If a sequence $(\g_n)$ in $\G$ converges to $\e$, then there exists 
a function $F: \N \to \N$ diverging to infinity and $n_1\in\N$ such that for all integers $n\ge n_1$, there exists a normal form of $\g_n$ containing $\omega_{F(n)}$ as a left subword.
\end{proposition}

We prove this result in \S\ref{sec:proofofaltthm}.

\begin{proposition}\label{prop:reln}
Fix  a word metric $d_\G$ on $\G$.
 For all $\omega\in\G$ satisfying $\rl(\omega)\ge 3$, there exists a constant $D = D(\omega)\ge0$ such that the following holds:
 If $\g\in\G$ is any element such that some normal form of $\g$ contains some normal form of $\omega$ as a left subword, then
\begin{enumerate}[(i)]
 \item in the case $\G = \G_A\star_\H\G_B$, we have $d_\G(\pr_{\G_A}(\g),1_\G) \le D$ and $d_\G(\pr_{\G_B}(\g),1_\G) \le D$.
 \item in the case $\G = \Pf$, we have $d_\G(\pr_\P(\g),1_\G) \le D$.
\end{enumerate}
\end{proposition}

Using this result, it can be shown that alternating sequences cannot have type I accumulation points in the boundary of $\G$.
See \S\ref{sec:prop:reln} for a proof of \Cref{prop:reln}.

\begin{notation}\label{rem:ddgdt}
We set up some notation for the rest of this section: We let $d_\G$ denote an arbitrary word metric on $\G$. Moreover, we reserve the notation $d$ to denote a word metric on $\G$ induced by a finite symmetric generating set $S$ of $\G$, where, (i) in the case of $\G = \G_A\star_\H\G_B$, $S$ is the union of some chosen finite symmetric generating sets of $\G_A$ and $\G_B$, and (ii) in the case of $\G = \Pf$, $S$ is the union of some chosen finite symmetric generating set of $\P$ and $\{f,f^{-1}\}$.
Recall that the identity map $(\G,d)\to(\G,d_\G)$ is a quasiisometry.
Finally, we reserve the notation $d_T$ to denote the distance function on the corresponding Bass-Serre tree $T$ induced by declaring that all the edges of $T$ are of unit length.
\end{notation}

\subsection{Proof of \Cref*{thm:alternating_one}}\label{sec:proofofaltthm}
The following result demonstrates the existence of an alternating sequence $(\omega_k)$ converging to $\e\in\vii\G$ in the statement of \Cref{thm:alternating_one}.

\begin{lemma}\label{prop:Dalternating}
There exists $D_0\ge0$ such that for every $\e\in\vii\G$, there exists a {\em $D_0$-alternating} sequence $(\omega_n)$ in $(\G,d_\G)$ converging to $\e$.
\end{lemma}

In the statement above, ``{$D_0$-alternating}'' means that $(\omega_n)$ is alternating and  lies within distance $D_0$ from any geodesic ray in $(\G,d_\G)$ emanating from $1_\G$ asymptotic to $\e$.

In the proof of \Cref{prop:Dalternating}, we work with the word-metric $d$; see \Cref{rem:ddgdt}.
The general case, i.e., when $d_\G$ is an arbitrary word metric, would then follow by applying the Morse lemma.

\begin{proof}[Proof of \Cref*{prop:Dalternating} in the case of amalgamated free products]
Let us consider a uniform quasigeodesic $c:\N\cup\{0\} \to (\G,d)$ emanating from $c(0)= 1_\G$ asymptotic to $\e$; such a ray is described by a sequence $(s_i)$ in $S$ such that for all $k\in\N$,
\[
 c(k) = s_1\cdots s_k.
\]
Note that the sequence $(c(k))$ cannot entirely lie in $\G_A\cup\G_B$ since, otherwise, the sequence $(c(k))$ would converge to a type I ideal boundary point.
Let $i_1$ be the largest number such that $c(i_1)$ lies in $\G_A\cup \G_B$.
For the same reason as above, the sequence $(c(i_1)^{-1} c(k))_{k\in\N}$ cannot lie in $\G_A\cup\G_B$.
Thus, let $i_2>i_1$ be the largest number such that $c(i_1)^{-1}c(i_2)$ lies in $\G_A\cup \G_B$.
Similarly, let $i_3>i_2$ be the largest number such that $c(i_2)^{-1}c(i_3)$ lies in $\G_A\cup\G_B$.
Proceeding inductively, we find a sequence $(c(i_{k})^{-1}c(i_{k+1}))_k$ which alternates between $\G_A$ and $\G_B$; the elements of that sequence are the letters for our alternating sequence $(\omega_k)$:
\[
 \omega_k = c(i_1) [c(i_1)^{-1}c(i_2)] \cdots [c(i_{k-1})^{-1}c(i_k)] = c(i_k).
\]
Since the sequence $(\omega_k)$ lies in the quasigeodesic ray $c$, it converges to $\e$. \end{proof}

\begin{proof}[Proof of \Cref*{prop:Dalternating} in the case of HNN extensions]
For $\e\in\vii\G$, pick any uniform quasigeodesic ray $c: \N\cup\{0\} \to \G$ that emanates  from $c(0)= 1_\G$ and is asymptotic to $\e$. Such a ray is described by a sequence $(s_i)$ in $S$ such that for all $k\in\N$,
\[
 c(k) = s_1\cdots s_k.
\]
We find an infinite  string
\begin{equation}\label{eqn:altstringone}
  \hat{\rm S} : \quad\p_0, f^{\epsilon_1}, \p_1, f^{\epsilon_1}, \p_2, \dots,
\end{equation}
which has the property that $\p_i\in\P$, $\epsilon_i=\pm1$, and, for every $n$, \begin{equation}\label{eqn:altstringtwo}
\rl(\p_0 f^{\epsilon_1}\p_1\cdots  f^{\epsilon_{n-1}}\p_{n-1}f^{\epsilon_n} ) = n.
\end{equation}
Let $i_0\in\N\cup\{0\}$ be the largest number such that $c(i_0)\in\P$; set $\p_0 = c(i_0)$.
Let $i_1\ge i_0$ be the largest number such that $c(i_0)^{-1}c(i_1)\in\{f,f^{-1}\}$; define $\epsilon_1$ in the  obvious way.
Let $i_2\ge i_1$ to be  the largest number such that $c(i_1)^{-1}c(i_2)\in\P$; set $\p_1 = c(i_1)^{-1}c(i_2)$.
Let $i_3\ge i_2$ be the largest number such that $c(i_2)^{-1}c(i_2)\in\{f,f^{-1}\}$; define $\epsilon_2$ accordingly.
We proceed inductively to obtain the above string.
It is a straightforward check that $\rl(\p_0 f^{\epsilon_1}\p_1\cdots  f^{\epsilon_{n-1}}\p_{n-1}f^{\epsilon_n} ) = n.$

We observe that if $\epsilon_i =-1$ and $\p_i\in\H_+$ (resp. $\epsilon_i=1$ and $\p_{i+1}\in\H_-$), then
\eqref{eqn:altstringtwo} forces $\epsilon_i = -1$ (resp. $\epsilon_i = 1$).
Thus, $(\omega_k)$, where $\omega_k \coloneqq \p_0 f^{\epsilon_1}\p_1\cdots  f^{\epsilon_{n-1}}\p_{n-1}f^{\epsilon_n} $, is an alternating sequence.
Since the sequence $(\omega_k)$ lies in the quasigeodesic ray $c$, it converges to $\e$.
\end{proof}

\medskip
The same construction used to prove \Cref{prop:Dalternating} can be applied to ``uniform quasigeodesic segments'' in $\G$ to show the following result:

\begin{definition}
For $D\ge 0$, a {\em $D$-normal form} of $\g\in\G$ satisfying $\rl(\g)\ge 1$ is a normal form of $\g$ such that all left subwords in that normal form lie in a $D$-neighborhood of any geodesic segment in $(\G,d_\G)$ connecting $1_\G$ and $\g$.
\end{definition}

\begin{lemma}\label{prop:good_form}
 There exists $D_0\ge 0$ such that every element $\g\in\G$ satisfying $\rl(\g)\ge 1$
 has a $D_0$-normal form.
\end{lemma}

Now, we prove \Cref{thm:alternating_one}.

\begin{proof}[Proof of \Cref*{thm:alternating_one}]
We give a proof in the amalgamated free product case; the HNN extension case is similar.

We continue with the proof of \Cref{prop:Dalternating} above (the amalgamated free product case).
Let $(\g_n)$ be any sequence in $\G$ converging to $\e\in\vii\G$.
Since $(\g_n)$ fellow travels the uniform quasigeodesic $c$, we may pick a divergent sequence $(t_n)$ in $\N\cup\{0\}$ and, for each $n$,
 a uniform quasigeodesic segment
\[
 c_n : [0,l_n]\cap \Z \to \G,\quad
 c_n(0) = 1_\G, c_n(l_n) = \g_n,
\]
such that for all $n\in\N$, $c_n\vert_{[0,t_n]\cap\Z} = c\vert_{[0,t_n]\cap\Z}$.
See \Cref{pic4} for an illustration of $c$ and $c_n$ in (the Cayley graph) of $\G$.
We apply the same procedure used in the proof of \Cref{prop:Dalternating}  on $c_n$ to yield 
a $D$-normal form for $\g_n$.

\begin{figure}[h]
\centering
\begin{overpic}[scale=1,tics=8]{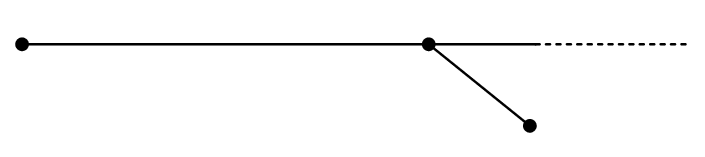}
\put(2,12){$1_\G$}
\put(76.5,3.9){$\g_n = c_n(l_n)$}
\put(99,15.5){$\e$}
\put(56.5,11.8){\small $c(t_n)$}
\put(80,18){$c$}
\end{overpic}
\caption{}
\label{pic4}
\end{figure}

\begin{lemma}
 There exists $n_1\in\N$ such that for all $n\ge n_1$, the prescribed normal form of $\g_n$ contains $\omega_1$ as a leftmost subword.
\end{lemma}

\begin{proof}
 We argue by contradiction: Suppose that the assertion is false. Then, a divergent sequence $(n_i)$ exists in $\N$ such that for all $i$,
 the leftmost letter $\lambda_i$ of the prescribed normal form of $\g_{n_i}$ is different from $\omega_1$.
 Since $c_{n_i}\vert_{[0,t_{n_i}]\cap\Z} = c\vert_{[0,t_{n_i}]\cap\Z}$, for all $i$ large enough
 it must hold that $\lambda_i = c_{n_i}(r_i)$, where $r_i> t_{n_i}$.
 However, $\lambda_i\in\G_A\cup\G_B$.
 Since $c_{n_i}$ are uniform quasigeodesics with $c_{n_i}(0)=1_\G \in\G_A\cup\G_B$, it follows that $c_{n_i}({[0,r_i]}\cap\Z)$, which contains $c({[0,t_{n_i}]}\cap\Z)$ as a subset, lies in a uniform neighborhood of $\G_A\cup\G_B$.
 However, since $t_{n_i}\to\infty$, $\rl(c(t_{n_i}))$ goes to infinity.
 Thus, $c(t_{n_i})$ cannot stay in a uniform neighborhood of $\G_A\cup\G_B$ (cf. \Cref{prop:reln}), yielding a contradiction.
\end{proof}

We now finish the proof of \Cref{thm:alternating_one} by induction.
Suppose that $c(i_1) = \omega_1$.
Since $c_n\vert_{[0,t_n]\cap\Z} = c\vert_{[0,t_n]\cap\Z}$, by the claim above, it holds that for all sufficiently large $n$, say $n\ge n_1$, $c_n(i_1) = c(i_1) = \omega_1$.
Applying the same argument to the quasigeodesic ray/segments
\[\begin{array}{ll}
 \bar c: [0,\infty)\cap \Z\to \G,&  c(t) = \omega_1^{-1}c(t+i_1),\\
 \bar c_n:[0,l_n-i_1]\cap\Z  \to \G,&  c_n(t) = \omega_1^{-1}c_n(t+i_1),
\end{array}\]
shows that there exists $n_2> n_1$ such that 
for all $n\ge n_2$, the prescribed normal form of $\g_n$ contains $\omega_2$ as a leftmost subword.
Arguing inductively, 
it follows that there exists an increasing sequence $(n_k)$ of natural numbers such that for all $n\ge n_k$, the prescribed normal form of $\g_n$ contains $\omega_k$ as a leftmost subword.
Thus, the desired function $F : \N \to \N$ in \Cref{thm:alternating_one}
can be defined as $F(n) \coloneqq k$, if $n\in[n_k,n_{k+1})$.
\end{proof}

\subsection{Proof of \Cref*{prop:reln}}\label{sec:prop:reln}
Let us first consider the case of amalgamated free products: For $\g\in \G = \G_A\star_\H\G_B$, consider a normal form of $\g$:
\[
 \g = \g_1\g_2\cdots \g_l.
\]
If $\g_1\in\G_A\setminus\H$, then normal form given above yields a finite sequence of points in $T$:
\begin{equation}\label{eqn:relnnormalform}
 \G_B,\quad \G_A,\quad \g_1 \G_B,\quad \g_{1}\g_2 \G_A,\quad \g_{1}\g_2\g_3 \G_B,\quad \dots,\quad \g_1\g_2\cdots \g_l \G_*,
\end{equation}
where $* = A$ if $l$ is even, or $* = B$ if $l$ is odd.
We observe that any two consecutive points in the above are adjacent vertices in $T$ (cf. \S\ref{sec:BStree}), and the sequence does not {\em backtrack}, i.e., for any point in \eqref{eqn:relnnormalform}, the vertices on the left and right to it are different:
For even $i$, let us examine the portion of the path
\[
 \g_{1}\cdots\g_{i-1} \G_B,\quad \g_{1}\cdots\g_i \G_A, \quad \g_{1}\cdots\g_{i+1} \G_B.
\]
Note that  $\g_i\in\G_B\setminus\H$ and $\g_{i+1}\in\G_A\setminus\H$.
Applying $(\g_{1}\cdots\g_{i})^{-1}$ to the above, we obtain
\[
 \G_B = \g_i^{-1}\G_B,\quad \G_A, \quad \g_{i+1} \G_B.
\]
However, since $\g_{i+1}\not\in\G_B$, $\G_B \ne \g_{i+1} \G_B$.
A similar analysis can be done when $i$ is odd.

Therefore, if we connect  each pair of consecutive vertices in \eqref{eqn:relnnormalform} by the unique edge in $T$ determined by the pair, we obtain a geodesic path in $T$.

\begin{lemma}\label{lem:reln}
If $\g \in\G = \G_A\star_\H\G_B$ lies in the coset represented by $v \in V(T)$, then $|d_T(v,\G_A) -\rl(\g)| \le 1$, $|d_T(v,\G_B) -\rl(\g)| \le 1$.
\end{lemma}
\begin{proof}
This claim is easily checked when $\g\in\H$. So, suppose that $\g\not\in\H$.
Let $\g_1\g_2\cdots \g_l$ be a normal form of $\g$ such that $\g_1\in\G_A\setminus \H$; the case $\g_1\in\G_B\setminus \H$ follows by a similar argument.
 Note that $v$ is one of the two rightmost entries in the sequence \eqref{eqn:relnnormalform}.
 Moreover, the discussion above shows that the sequence \eqref{eqn:relnnormalform} is a geodesic sequence of vertices in $T$.
 Thus, $\rl(\g) \ge d_T(v,\G_A) \ge\rl(\g) -1$ and $\rl(\g)+1\ge d_T(v,\G_B) \ge \rl(\g)$.
\end{proof}

Similar discussion holds in the case of  HNN extensions:
For $\g\in\G = \Pf$, let 
\[
 \g = \p_{0} f^{\epsilon_{1}}\p_1 \cdots f^{\epsilon_n} \p_n
\]
be a normal form of $\g$.
The above normal form produces a finite sequence in $T$:
\[
 \P, \quad 
 \p_{0} f^{\epsilon_{1}}\P, \quad
 \p_{0} f^{\epsilon_{1}}\p_1 f^{\epsilon_2}\P, 
 \quad\dots, \quad 
 \g \P = \p_{0} f^{\epsilon_{1}}\p_1 \cdots f^{\epsilon_n} \P.
\]
Similarly to the amalgamated free product case discussed above, one can check that the above sequence does not backtrack and that consecutive vertices in the path are adjacent in $T$.
This yields the following:

\begin{lemma}\label{lem:relnHNN}
 For all $\g\in\G = \Pf$,  $d_T(v,\P) = \rl(\g)$.
\end{lemma}
\begin{proof}
 The proof is similar to that of \Cref{lem:reln}. We omit the details.
\end{proof}

We prove \Cref{prop:reln}.

\begin{proof}[Proof of \Cref*{prop:reln}]
We discuss the proof in the case of amalgamated free products $\G = \G_A\star_\H\G_B$; the case of HNN extensions $\G = \Pf$ is similar.

 Consider a normal form of $\g$ which contains some normal form of $\omega$ as a left subword:
 \[
  \g = \underbrace{\g_1\cdots\g_m}_{=\omega}\g_{m+1}\cdots \g_l.
 \]
 We further assume that $\g_1\in\G_A\setminus \H$; the case $\g_1\in\G_B\setminus \H$ follows similarly.
 The above normal form induces a geodesic path in $T$ (see the paragraph before \Cref{lem:reln}) from the vertex $\G_B$ to the vertex $v_\g \coloneqq \g\G_*$, where $* = A$ if $l$ is even, or $* = B$ if $l$ is odd.
  This path also contains the vertex $v_\omega \coloneqq \omega\G_*$,  where $* = A$ if $m$ is even, or $* = B$ if $m$ is odd.
 The vertex $\G_A$ also lies in that path, see \eqref{eqn:relnnormalform}.
 Thus, any path in $T$ connecting $v_\g$ to $\G_A$ or $\G_B$ must contain $v_\omega$.
 
 For the rest of the proof, let $*$ denote either $A$ or $B$.
  Let $c : [0,n]\cap \Z \to \G$ be a shortest geodesic in $(\G,d)$\footnote{See \Cref{rem:ddgdt} for our notation.} such that $c(0)\in\G_*$ and $c(n) = \g$.
  By definition, $c(0)$ is a closest point in $\G_*$ to $c(k)$, for any $k$ in the domain of $c$.
  For $k\in [0,n]\cap\Z$, let $\bar c(k)\coloneqq \{c(k)\G_A,c(k)\G_B\}$.
  
\begin{claim}
 For $k= 0,\dots,n-1$, we have that
 $\bar c(k)\cap\bar c(k+1) \ne \emptyset$.
\end{claim}

\begin{proof}
 We observe that $c(k+1) = c(k) \cdot s_k$, for some generator $s_k\in S \subset \G_A\cup \G_B$.
 Let $c(k) = \tg_1\cdots\tg_r$
 be a normal form of $c(k)$.
 Assume that $\tg_r \in\G_A$; the other possibility $\tg_r \in\G_B$ can be similarly treated.
 If $s_k\in\G_A$, then $c(k)\G_A = \tg_1\cdots\tg_{r-1} \G_A \in \bar c(k+1)$ since $c(k+1) = \tg_1\cdots\tg_{r-1} (\tg_r s_k)\in \tg_1\cdots\tg_{r-1} \G_A$.
 Similarly, if $s_k\in\G_B$, then $c(k) \G_B \in \bar c(k+1)$.
 The claim follows.
\end{proof}

Let $V' = \bigcup_{k=0}^n \bar c(k)\subset V(T)$. Consider the induced subtree of $T'\subset T$ determined by $V'$.
By the claim above, $T'$ is connected.
Since $c(0)\in \G_*$, we get that $\G_*\in V'$.
Obviously, $v_\g\in V'$ as well.
Since $T'$ is connected, by the first paragraph of this proof, $v_\omega\in V'$.
Therefore, there exists some $k_0$ in the domain of $c$ such that $v_\omega\in \bar c(k_0)$.
In other words, $c(k_0)$ lies in (the coset represented by) $v_\omega$.
Since $c(0)$ is also a closest point in $\G_*$ to $c(k_0)$, it follows that $c(0)$ is uniformly close to $\pr_{\G_*}(v_\omega)$.
However, since $\rl(\omega)\ge 3$, by \Cref{lem:reln}, $d_T(v_\omega,\G_*) \ge 2$.
Thus, by \eqref{boundary_disjoint},\footnote{Note that $\geo\G_{v_\omega} = \geo v_\omega$.} $\geo\G_* \cap \geo v_\omega =\emptyset$ and, hence,
 $\pr_{\G_*}(v_\omega)$ is bounded in $\G_*$ (see \Cref{cor:proj_fellow_travel}).
 This completes the proof of this result when $d_\G = d$.
 
 In the general case, i.e., when $d_\G$ is an arbitrary word metric, then the result follows by the fact that the identity map $(\G,d)\to(\G,d_\G)$ is a quasiisometry.
\end{proof}

\section{Preliminaries on discrete isometry groups of symmetric spaces}\label{sec:prelim}

We recall some preliminary facts on symmetric spaces, mainly to set up some notations.
Then, we refer to \cite[Appendix 5]{MR823981} for a quick discussion and to \cite{eberlein} for a more detailed exposition.

Let $G$ be a real semisimple Lie group of noncompact type with a finite center.
We impose some mild assumptions on $G$; see \Cref{rem:assumptionG} below.
Let $X = G/K$ denote the (globally) symmetric space of $G$, where $K$ is a maximal compact subgroup of $G$.
Then, $X$ is a nonpositively curved $G$-homogeneous space such that $X$ has no compact or Euclidean de Rham factors, i.e. $X$ is a symmetric space of {\em noncompact type}. 

The ideal boundary $\geo X$ is the set of equivalence classes of asymptotic rays in $X$ on which $G$ acts naturally.
The point stabilizers in $G$ of the action $G\acts \geo X$ are called {\em parabolic subgroups} of $G$.
The ideal boundary $\geo X$ carries a natural $G$-invariant spherical building structure, called the {\em Tits building} of $X$, and denoted by $\tits X$.
The top-dimensional simplices in $\tits X$ are called {\em chambers}.

\medskip
In this paper, we impose the following assumption on $G$, which are standing assumptions in the papers of Kapovich-Leeb-Porti \cite{MR3811766,MR3736790,KLP:Morse} we rely upon in this work:

\begin{assumption}\label{rem:assumptionG}
The group $G$ has a finitely many connected components such that for the associated symmetric space $X = G/K$, the spherical Tits building $\tits X$ is {\em thick}: That is, every {\em panel} (i.e., a codimension one simplex) in $\tits X$ is contained in at least three distinct  chambers. 
\end{assumption}

\begin{remark}
 \Cref{rem:assumptionG} is satisfied if $G$ is connected, for instance, when $G = {\rm SL}(n, \mathbb{R})$. However, there are cases where this condition is not met. For instance, if $X = {\rm SL}(3, \mathbb{R})/{\rm SO}(3)$ and $G$ is the {\em full} isometry group of $X$, the associated spherical Tits building is not thick. Such semisimple Lie groups are excluded from this theory.
\end{remark}

Each chamber in $\geo X$ is also a fundamental domain for the action $G \acts \geo X$.
The stabilizer in $G$ of a chamber $\sigma$ (which is the same as the stabilizer in $G$ of any interior point of $\sigma$) is called a {\em minimal parabolic subgroup} of $G$.
The group $G$ also acts on the set of chambers in $\tits X$ transitively, so any two minimal parabolic subgroups are conjugate. Therefore, the space of all chambers in $\tits X$ can be identified with the analytic manifold
\[
 \Fs \coloneqq G/P_{\smod},
\]
where $\smod$ is a chosen chamber in $\tits X$ and $P_{\smod}$ is the minimal parabolic subgroup stabilizing $\smod$.

More generally, for a face $\nmod \subset\smod$,
the space of  simplices $\nu$ in $\tits X$ of {\em type $\nmod$}, i.e., those simplices $\nu$ in $\tits X$ which can be brought to $\nmod$ by the action $G\acts\tits X$, can be identified with the analytic manifold
\[
 \Fn \coloneqq G/P_{\nmod},
\]
where $P_{\nmod}$ is the parabolic subgroup of $G$ stabilizing $\nmod$.

A pair of simplices $\nu_\pm$ in $\tits X$ is called {\em antipodal} if there is a complete geodesic line in $X$, which is forward (resp. backward) asymptotic to some interior point of the simplex $\nu_+$ (resp. $\nu_-$).
If $\nmod^\pm \subset \smod$ denote the types of $\nu_\pm$, then they satisfy
\[
 \nmod^+ = \iota (\nmod^-),
\]
where $\iota : \smod\to\smod$ denotes the {\em opposition involution}.
For a simplex $\nu_-$ in $\tits X$ of type $\nmod^-$,
the set of simplices in $\tits X$ (regarded as points in $\mathrm{Flag}({\nmod^+})$, where $\nmod^+ = \iota (\nmod^-)$) antipodal to $\nu_-$, 
\[
 C(\nu_-) \coloneqq
 \{
  \nu_+ \in \mathrm{Flag}({\nmod^+}) \mid
  \nu_+ \text{ is antipodal to } \nu_-
 \} \subset \mathrm{Flag}({\nmod^+}),
\] 
is an open dense $P_{\nmod^-}$-homogeneous subset of $\mathrm{Flag}({\nmod^+})$.

\begin{definition}[Antipodality]\label{def:antipodality}
 A pair $(A,B)$, where $A \subset \mathrm{Flag}({\nmod^+})$ and $B \subset \mathrm{Flag}({\nmod^-})$, is said to be {\em antipodal} to each other (or, $A$ is {\em antipodal} to $B$)
 if, for all $\nu_+\in A$ and all $\nu_-\in B$, $\nu_+$ is antipodal to $\nu_-$.
\end{definition}

\subsection{Regular sequences}\label{sec:regularsequences}

Let $\nmod\subset \smod$ be a face.
A sequence $(g_n)$ in $G$ is said to {\em $\nmod$-converge} to a point $\nu\in \Fn$, which is denoted by
\[
 g_n \tof \nu,
\]
if every subsequence of $(g_n)$ has a further subsequence $(g_{n_k})$ such that there exists $\nu_- \in {\rm Flag}(\iota\nmod)$  such that
\[
 {g_{n_k}}\vert_{C(\nu_-)} \to \nu,\quad
 \text{uniformly on compacts.}
\]
In this situation, $(g_n)$ is called a {\em $\nmod$-convergent} sequence in $G$ and $\nu\in\Fn$ is called the {\em $\nmod$-limit point} of $(g_n)$.
For a $\nmod$-convergent sequence $(g_n)$ in $G$, its $\nmod$-limit point $\nu$ is unique.
See \cite[\S 4]{MR3736790} for more details.

For a discrete subgroup $\G$ of $G$, the set of all $\nmod$-limit points in $\Fn$ of $\nmod$-convergent subsequences of $\G$ is called the {\em $\nmod$-limit set} of $\G$, which is denoted by
\[
 \Lambda_{\nmod}(\G).
\]
The limit set $\Lambda_{\nmod}(\G)$ is a $\G$-invariant compact subset of $\Fn$; however; it may be empty, even if $\G$ is infinite.

A sequence $(g_n)$ in $G$ is {\em $\nmod$-regular} if every subsequence of $(g_n)$ contains a $\nmod$-convergent subsequence.
Moreover, if $(g_n)$ is $\nmod$-regular, then the inverse sequence $(g_n^{-1})$ is $\iota\nmod$-regular.
Similarly, a subgroup $\G$ of $G$ is called  {\em $\nmod$-regular}
if every sequence in $\G$ is $\nmod$-regular.
Such subgroups are necessarily discrete.
Clearly, $\nmod$-regular subgroups are also $\iota\nmod$-regular.

The following results help verify the regularity and flag-convergence of certain sequences considered in this paper.

Let $(g_n)$ be a sequence in $G$.

\begin{lemma}\label{prop:regulairty}
 If there exists a compact subset $A \subset \Fn$ with a nonempty interior such that the sequence $(g_n A)$ of compact subsets of $\Fn$ converges to a point $\nu \in \Fn$, then $(g_n)$ is $\nmod$-regular and $g_n \tof \nu$.
\end{lemma}

See \cite[Lemma 1.10]{DK22}.

Let $d$ be a distance function on $\Fn$ compatible with the manifold topology of $\Fn$.
Let $(A_n)$ a sequence of nonempty subsets of $\Fn$.
We say that $(A_n)$ {\em shrinks},
if the diameter of $A_n$ converges to zero as $n\to \infty$.
Moreover, we say that $(A_n)$ {\em converges} a point $\nu\in\Fn$, written as $A_n\to\nu$, if the the diameters of the subsets $A_n\cup\{\nu\}$ converge to zero as $n\to \infty$.
Since $\Fn$ is compact, these notions do not depend on the chosen distance function.

\begin{lemma}\label{prop:regulairty_two}
 If there exists a compact subset $A\subset \Fn$ with a nonempty interior such that $(g_n A)$ shrinks, 
 then $(g_n)$ is $\nmod$-regular.
\end{lemma}

See \cite[Corollary 1.11]{DK22}.

The next result can be extracted from the proof of 
\cite[Lemma 3.4]{DK22}.
For completeness, we give a proof,

\begin{lemma}\label{lem:mainlemma}
 If there exist compact subsets $B,B'\subset\Fn$ with nonempty interior such that $B' \subset B^\o$, $B\subset C(\nu)$ for some $\nu\in\Fn$, and for all $n\in\N$,
 $
  g_{n+1}g_n^{-1} (B) \subset B',
 $
 then $(g_n)$ is $\nmod$-regular.
\end{lemma}

\begin{proof}
After replacing the sequence $(g_n)$ by $(h_n)$, where $h_n\coloneqq g_n g_1^{-1}$, we have that  
\begin{equation}\label{eqn:one:mainlemma}
 h_n(B)\subset B \quad\text{and}\quad h_{n+1}h_n^{-1} (B) \subset B'.
\end{equation}

  Suppose, to the contrary, that $(g_n)$ is not $\nmod$-regular. Then $(h_n)$ is also not $\nmod$-regular.
 Moreover, by the right side of \eqref{eqn:one:mainlemma}, it follows that the sequence $(h_n)$ has no accumulation points in $G$.
  Therefore, after extraction, $(h_n)$ is {\em $\eta_{\rm mod}$-pure} for some face $\eta_{\rm mod} \subset \smod$ such that $\nmod\not\subset\eta_{\rm mod}$.
 By \cite[Proposition 9.5]{MR3811766}, after further extraction, there exist $\eta_+\in{\rm Flag}(\eta_{\rm mod})$ and $\eta_-\in{\rm Flag}(\iota\eta_{\rm mod})$, and a surjective algebraic map $\phi : C_{\rm Fu}(\eta_-) \to \St(\eta_+)$ such that
 $h_{n_k}\vert_{C_{\rm Fu}(\eta_-)}  \to \phi$
uniformly on compacts.
Here $\St(\eta_+) \coloneqq \pi^{-1}(\eta_+)$ and $C_{\rm Fu}(\eta_-) \coloneqq \pi^{-1}(C(\eta_-))$, where $\pi$ denotes the unique  $G$-equivariant projection $\Fs \to {\rm Flag}(\eta_{\rm mod})$ mapping $\smod\mapsto \eta_{\rm mod}$.
 
 By \eqref{eqn:one:mainlemma}, we obtain that $h_{n_k}{\tilde B} \subset {\tilde B}$, where
 \[
  {\tilde B} = \bigcup_{\nu\in B} \St(\nu).
 \]
 Consider any point $\eta\in C(\eta_-) \subset {\rm Flag}(\eta_{\rm mod})$ such that $\St(\eta)\cap ({\tilde B})^\o$ is nonempty.\footnote{Note that since $C_{\rm Fu}(\nu_-)$ is open dense in $\Fs$ and ${\tilde B}$ has nonempty interior in $\Fs$, $C_{\rm Fu}(\nu_-)\cap \tilde B$ is nonempty.} 
  By \cite[Lemma 1.8]{DK22},
\begin{equation}\label{eqn:mainlemma}
   \lim_{k\to\infty} \max_{\sigma\in\St(h_{n_k}\eta)} d(h_{n_k}h_{n_{k-1}}^{-1}\sigma,\sigma) = 0.
\end{equation}

 Since, for all $k\in\N$, $h_{n_k}({\tilde B})^\o \subset ({\tilde B})^\o$,
 $\St(h_{n_k}\eta)\cap ({\tilde B})^\o \ne\emptyset$.
 Moreover, since $\nmod\not\subset\eta_{\rm mod}$, by \cite[Lemma 1.3]{DK22}, $\St (h_{n_k}\eta)$ is not contained in $C_{\rm Fu}(\nu)$ for any $\nu\in \Fn$.
 As $B\subset C(\nu)$ for some $\nu\in \Fn$, it follows that $\St (h_{n_k}\eta)\not\subset \tilde B$.
 However, by \cite[Lemma 1.2]{DK22}, $\St(h_{n_k}\eta)$ is connected.
 So, $\St(h_{n_k}\eta)$ intersects $\partial {\tilde B}$.
 Pick $\sigma_k \in \St(h_{n_k}\eta)\cap \partial {\tilde B}$ for each $k\in\N$.
 By \eqref{eqn:mainlemma}, we have that
\begin{equation}\label{eqn:two:mainlemma}
   \lim_{k\to\infty} d(h_{n_k}h_{n_{k-1}}^{-1}\sigma_{k},\sigma_{k}) = 0.
\end{equation}
 
Since $h_{n+1}h_n^{-1} (B) \subset B'$ (by \eqref{eqn:one:mainlemma}), we have that
 $h_{m}h_n^{-1} (B) \subset B'$ for all $m>n$.
 In particular, $h_{n_k}h_{n_{k-1}}^{-1} (B) \subset B'$, for all $k\in\N$, which implies that
 \[
  h_{n_k}h_{n_{k-1}}^{-1} (\tilde B) \subset \tilde B' \coloneqq \bigcup_{\nu\in B'} \St(\nu).
 \]
 Since $\sigma_k\in \tilde B$, by above, we get $h_{n_k}h_{n_{k-1}}^{-1}(\sigma_k)\in \tilde B'$.
 However, $\tilde B' \subset \tilde B^\o$, which shows that
 \[
  d(h_{n_k}h_{n_{k-1}}^{-1}(\sigma_k), \sigma_k) \ge d(\tilde B', \partial\tilde B) >0, \quad
  \text{for all } k\in\N,
 \]
 contradicting \eqref{eqn:two:mainlemma}.
\end{proof}

\begin{lemma}\label{lem:regularity_four}
 Suppose that $(g_n)$ is $\nmod$-regular.
 If there exist compact subsets $A,B \subset \Fn$ with nonempty interior such that for all $n\in\N$, $g_n B\subset A$,
  then all the $\nmod$-limit-points of $(g_n)$ lie in $A$.
\end{lemma}

\begin{proof}
 Let $\nu\in\Fn$ be a $\nmod$-limit point of $(g_n)$. Then, there exists $\nu_-\in {\rm Flag}(\iota\nmod)$ and a subsequence $(g_{n_k})$ such that $g_{n_k}\vert_{C(\nu_-)}$ converges to the constant map $C(\nu_-) \to \{\nu\}$ uniformly on compacts, as $k\to\infty$.
 Let $\nu_1\in C(\nu_-) \cap B$.
 Then, the sequence $(g_{n_k} \nu_1)$, which lies in $A$, converges to $\nu$.
 Thus, $\nu\in A$.
\end{proof}

\begin{lemma}\label{lem:regularity_two}
 Suppose that 
 $g_n \tof \nu \in\Fn$.
 Let $A_{-}\subset {\rm Flag}(\iota\nmod)$  be a subset containing all the 
 $\iota\nmod$-limit points of $(g_n^{-1})$.
 If $A \subset \Fn$ is any compact subset antipodal to $A_{-}$,
 then $(g_n A)$ converges to $\{\nu\}$.
\end{lemma}

\begin{proof}
 Suppose not.
 Then, there exists a sequence $(\nu_k)$ in $A$ and a subsequence $(g_{n_k})$ of $(g_n)$ such that $(g_{n_k}\nu_k)$ converges to some point $\nu'\in\Fn$ different from $\nu$.
 After further extraction of $(g_{n_k})$, we may assume that $(g_{n_k}^{-1})$ $\iota\nmod$-converges to some point $\nu_-\in A_-$.
 Since, by hypothesis, $A \subset C(\nu_-)$, by 
 \cite[Lemma 4.18]{KLP:Morse}, we
 get that
 $g_{n_k} (A) \to \nu$,  
 as $n\to\infty$.
 Since, for each $k$, $\nu_k\in A$, we also obtain that $g_{n_k}\nu_k \to \nu$, which implies that $\nu' = \nu$. This is a contradiction.
\end{proof}

The following result follows from \Cref{lem:regularity_two}:

\begin{lemma}\label{cor:regulairty}
  Suppose that $(g_n)$ is $\nmod$-regular. 
  Let $A_-\subset  {\rm Flag}(\iota\nmod)$  be a subset containing all the $\iota\nmod$-limit points of $(g_n^{-1})$.
 If $A \subset \Fn$ is any compact subset antipodal to $A_-$,
 then $(g_n A)$ shrinks.
\end{lemma}

\subsection{Anosov subgroups}\label{sec:anosov_subgroups}
We fix once and for all an $\iota$-invariant face $\tmod$  of $\smod$.
We focus on the special class of discrete subgroups of $G$, called {\em $\tmod$-Anosov subgroups}.
This class of subgroups has several different characterizations.
For our purpose, we use the following characterization of $\tmod$-Anosov subgroups:

\begin{definition}[Asymptotically embedded]\label{def:asymptotically_embedded}
 A subgroup $\G$ of $G$ is {\em $\tmod$-asymptotically embedded} if
\begin{enumerate}[(i)]
 \item $\G$, as an abstract group, is hyperbolic, and
 \item there exists a $\G$-equivariant {antipodal map}\footnote{That is, $\xi$ maps every distinct pair of points to a  pair of antipodal points.} $\xi : \geo \G \to \Ft$, which preserves the {\em convergence dynamics}: That is, for every sequence $(\g_n)$ in $\G$ and every point $\e \in \geo \G$,
\begin{equation}\label{eqn:limitmap_continuity}
   \text{if } \g_n\toC \e, 
   \text{ then } \g_n \tof \xi(\e).
\end{equation}
\end{enumerate}
\end{definition}

The image of the map $\xi$ in the above definition is precisely $\LT{\G}$, the limit set of $\G$ in $\Ft$: Clearly, the image of $\xi$ is contained in $\LT{\G}$.
To show that opposite inclusion, let $\tau\in\LT{\G}$.
By definition, there exists a sequence $(\g_n)$ in $\G$ such that $\g_n\tof \tau$.
After extraction, $(\g_n)$ converges in $\bar\G$ to some point $\e\in\geo\G$.
By \eqref{eqn:limitmap_continuity}, $\xi(\e) = \tau$.

\begin{remark}\label{rem:asmp}
The above definition of $\tmod$-asymptotically embedded subgroups, which is a slight variation of the one given by Kapovich-Leeb-Porti in \cite[Definition 5.12]{MR3736790}, 
is more straightforward to verify  in this paper due to a certain ``ping-pong'' type arguments we use here.
 The equivalence between these two definitions can be checked as follows (cf. \cite[\S 5]{DK22}):
 
Suppose that $\G<G$ is  $\tmod$-asymptotically embedded in the sense of \Cref{def:asymptotically_embedded}.
By the first condition (i), $\G$ is hyperbolic.

We first verify that $\G$ is $\tmod$-regular: If $(\g_n)$ is any sequence in $\G$ without repeated entries, then, after extraction, $\g_n\toC \e$, for some $\e\in\geo\G$.
Thus, by condition (ii), $(\g_n)$ $\tmod$-converges to $\xi(\e)$.
Thus, $(\g_n)$ is $\tmod$-regular.

We next verify that $\xi$ is an embedding. Indeed, since $\geo \G$ is a compact and $\Ft$ is Hausdorff, it is enough to establish that $\xi$ is an injective continuous map: 
 Since $\xi$ is an antipodal map, $\xi$ is injective.
 To verify continuity, pick $x\in X$ and
 consider the map
\begin{equation*}
   \phi : \bG \to {\bar X}^\tmod = X\sqcup \Ft,
\end{equation*}
 whose restriction to $\geo\G$ is $\xi$ and whose restriction to $\G$ is the orbit map $\g \mapsto \g \cdot x$.
 The space $X\sqcup \Ft$ is topologized with the topology of $\tmod$-convergence. The restriction of this topology to $X$ or $\Ft$ coincides with their respective manifold topologies.
 By \eqref{eqn:limitmap_continuity}, $\phi$ is continuous and, hence, so is the restriction $\xi = \phi\vert_{\geo \G}$.
 
 Therefore, $\xi$ is a $\G$-equivariant homeomorphism between $\geo\G$ and its image, $\LT{\G}$.
 Finally, since $\xi$ is an antipodal map, $\LT{\G}$ is an antipodal subset of $\Ft$.
 Therefore, $\G$ is {\em $\tmod$-antipodal} in the sense of \cite{MR3736790}.
 Hence, $\G$ is $\tmod$-asymptotically embedded in the sense of \cite[Definition 5.12]{MR3736790}.
 
 The other direction is straightforward.
\end{remark}

 \subsection{Interactive pairs and triples}\label{sec:interactive}

Let $\tmod$ be an $\iota$-invariant face of $\smod$.
Following Maskit \cite{maskit:book}, we define the notion of ``interactive pairs'' and ``interactive triples."

\subsubsection{Interactive pairs}

Let $\G_A$ and $\G_B$ be discrete subgroups of $G$, and let $\H \coloneqq \G_A\cap \G_B$.
We denote this data by the triple $(\G_A,\G_B;\, \H)$.

\begin{definition}[Interactive pair]\label{defn:ping-pong}
A pair $(A,B)$ of compact subsets of $\Ft = G/P_\tmod$ is called an {\em interactive pair} for $(\G_A,\G_B;\, \H)$ if the following conditions are satisfied:
\begin{enumerate}[(i)]
 \item The interiors $A^\o$ of $A$ and $B^\o$ of $B$ are nonempty and disjoint.
 \item  $\H$ leaves the sets $A$ and $B$ {\em precisely invariant}, i.e., $\H A = A$, $\H B = B$, and, for all elements $\alpha\in \G_A \setminus \H$ and $\beta\in \G_B \setminus \H$, we have that $\alpha {B} \subset A^\o$ and $\beta {A} \subset  B^\o$.
\end{enumerate}
\end{definition}

See \Cref{pic2} for an illustration.

\begin{proposition}\label{prop:faithful_amalgam}
 If $(\G_A,\G_B;\, \H)$ admits an interactive pair $(A,B)$ in $\Ft$,
 then the natural homomorphism
 \[
  \rho: \G_A \star_\H \G_B \to \< \G_A,\G_B\> <G
 \]
 from the abstract amalgamated free product $\G_A \star_\H \G_B$ to $G$ in injective.
 In particular, the subgroup $\G \coloneqq \< \G_A, \G_B\>$ of $G$ is naturally isomorphic to $\G_A \star_\H \G_B$.
\end{proposition}

\begin{proof}
 Let $\g\in\G_A \star_\H \G_B$ be any nontrivial element.
 We want to show that $\rho(\g)$ is nontrivial.
 Clearly, if $\g\in(\G_A\cup\G_B) \setminus \{1_\G\}$, then $\rho(\g) = \g$ and, hence, $\rho(\g)$ is nontrivial.
 So, we may assume that $\g\not\in \G_A\cup\G_B$.
 Choose a normal form $\g = \g_1 \cdots \g_l$ of $\g$ (see \S\ref{sec:reduced}).
 Since $\g\not\in\G_A\cup\G_B$, we have $l\ge 2$.
 Assume that $\g_l\in\G_B$ (the other possibility that $\g_l\in\G_A$ can be similarly analyzed).
 Then, $\rho(\g)A \subset A^\o\cup B^\o$ and, in particular, $\rho(\g)A \ne A$.
 Thus, $\rho(\g)$ is a nontrivial element of $G$.
\end{proof}

Compare with \cite[Theorem VII.A.10]{maskit:book}.

\begin{remark}
 With a little more effort, one can also show that the image of the homomorphism $\rho$ 
 in \Cref{prop:faithful_amalgam} is discrete.
 However, we do not need this result to prove our main theorems.
\end{remark}

\subsubsection{Interactive triples}

Let $\P$ be a discrete subgroup of $G$, and let $f\in G$. 
Define 
 \[
 \H_- \coloneqq ({f^{-1}\P f}) \cap \P, \quad
 {\H_+} \coloneqq \P \cap (f\P f^{-1}).\]
 Clearly, $f\H_- f^{-1} = {\H_+}$.
 We denote this data by the quadruple $(\P;\,\H_\pm;\,f)$.

\begin{definition}[Interactive triples]\label{def:interactive_triple}
 A triple $(A, B_\pm)$ of compact subsets  of $\Ft$
 is called an {\em interactive triple} for $(\P;\,\H_\pm;\,f)$ if the following conditions are met:
 \begin{enumerate}[(i)]
 \item The interiors $A^\o$, $B^\o_-$, $B_+^\o$ of $A$, $B_-$, $B_+$, respectively, are nonempty and pairwise disjoint.
 Furthermore, $B_-\cap B_+ = \emptyset$.

  \item $\H_\pm$ leaves $B_\pm$ {\em precisely invariant}, i.e.,
   $\H_\pm B_\pm = B_\pm$ and 
   $\p(B_\pm) \subset A^\o$ for all $\mu\in\P\setminus \H_\pm$. 
  \item  $f^{\pm1}(A) \subset B_\pm$ and $f^{\pm1} (B_\pm) \subset B_\pm^\o$.
 \end{enumerate}
\end{definition}

See \Cref{pic3}  for an illustration.

\begin{proposition}\label{prop:faithful_HNN}
 If $(\P;\,\H_\pm;\,f)$ admits an interactive triple $(A, B_\pm)$ in $\Ft$,
 then the natural homomorphism
 \[
  \Pf \to G
 \]
 is injective. In particular, the subgroup $\G \coloneqq \< \P,f\>$ of $G$ is naturally isomorphic to the HNN extension of $\P$ by the isomorphism $\phi :\H_- \to{\H_+}$ is given by $\phi(\eta) = f\eta f^{-1}$, for all $\eta\in\H_-$.
\end{proposition}

\begin{proof}
 The proof of this result is similar to the one of \Cref{prop:faithful_amalgam}.
 Hence, we omit the details.
\end{proof}

Compare with \cite[Theorem VII.D.12]{maskit:book}.

\section{Proof of \Cref{mainthm:amalgam}} 
\label{sec:CT_A}

In this section, we work under the hypothesis of \Cref{mainthm:amalgam}.
To simplify our situation, we replace $A$ and $B$ by the closure of their respective interiors.
It is easy to see that $(\cl(A^\o),\cl(B^\o))$
is still an interactive pair for $(\G_A,\G_B;\,\H)$.
However, the main advantage of the interactive pair $(\cl(A^\o),\cl(B^\o))$ is a stronger antipodality
hypothesis:

\begin{lemma}\label{lem:replacei}
$\cl(A^\o)$, resp. $\cl(B^\o)$, is  antipodal to $B^\o$, resp. $A^\o$.
\end{lemma}

\begin{proof}
Recall that $A^\o$ is antipodal to $B^\o$ under the hypothesis of \Cref{mainthm:amalgam}.
Let $\tau_+\in \cl(A^\o)$ and $\tau_-\in B^\o$ be any points.
We show that $\tau_\pm$ are antipodal.

Fix an auxiliary distance function $d$ on $\Ft$ compatible with its manifold topology. For $\tau\in\Ft$, let $\E(\tau)$ denote the  (compact) subset of $\Ft$ consisting of all points which are not antipodal to $\tau$.
Let $(\tau_n)$ be a sequence in $A^\o$ converging to $\tau_+$ and let $B_\epsilon(\tau_-)$ be a closed metric ball in $\Ft$ centered at $\tau_-$ of radius $\epsilon>0$ small enough such that
$B_\epsilon(\tau_-) \subset B^\o$.
By hypothesis, $\tau_n$ and $B^\o$ are antipodal. Hence, $d(\E(\tau_n),B_\epsilon(\tau_-)) > 0$, implying that $d(\E(\tau_n),\tau_-) \ge \epsilon$.
Since the Hausdorff distance between $\E(\tau_n)$ and $\E(\tau_+)$ converges to zero as $n\to\infty$,
it follows that $d(\E(\tau_+),\tau_-)\ge \epsilon>0$.
So, $\tau_-\not\in\E(\tau_+)$.
\end{proof}

\begin{assumption}\label{rem:replace_hypothesis_amalgam}
After replacing $(A,B)$ by  $(\cl(A^\o),\cl(B^\o))$ in the hypothesis of \Cref{mainthm:amalgam}, in place of (i), we will assume the following:
{\em
\begin{enumerate}[(i)]
 \item [(i)'] The pairs of subsets $(A,B^\o)$ and $(A^\o,B)$ of $\Ft$ are antipodal to each other.
\end{enumerate}}
This replacement does not affect the conclusion of \Cref{mainthm:amalgam}.
\end{assumption}

In \S\ref{sec:regularity_amalgam}, we will take additional advantage given by hypothesis (i)' above; see, for instance, the proof of \Cref{lem:shrinks}.

\begin{lemma}\label{lem:ping-pong}
Under the hypothesis of \Cref{mainthm:amalgam}
 \[
  \LT{\G_A} \subset  A, \quad
  \LT{\G_B} \subset  B, \quad \text{and} \quad
  \LT{\H} \subset A\cap B.
 \]
\end{lemma}

\begin{proof}
 It will be enough to show that $\LT{\G_A} \subset  A$.
 
 If $\H = \G_A$, then $\G_A$ preserves $A$, and since $A$ has a nonempty interior, all $\tmod$-limit points of $\G$ must lie in $ A$ (see \Cref{lem:regularity_four}).
 
 So, we can assume that $\H$ is a proper subgroup of $\G_A$, i.e., there exists some element $\a\in\G_A$ which is not an element of $\H$.
 For any point $\tau_+\in\LT{\G_A}$, consider a sequence $(\a_n)$ in $\G_A$ such that $\a_n \tof \tau_+$. 
 We may (and will) also assume that no elements of $(\a_n)$ lie in $\H$:
 Indeed, for every $n\in\N$, if $\a_n\in\H$, then replace the $n$-th entry $\a_n$ in $(\alpha_n)$  by $\a_n\a$.
 After replacing all such entries, the resulting sequence, again denoted by $(\a_n)$, does not share any elements with $\H$ but still satisfies $\a_n \tof \tau_+$.
 
Since for all $n\in\N$, $\alpha_n B \subset A$, \Cref{lem:regularity_four} implies that $\tau_+\in A$.
\end{proof}

\begin{remark}\label{rem:relaxation_amalgam}
 If $\LT{\G_A}\cap \partial A = \LT{\H}$, then, by \Cref{lem:ping-pong}, $\LT{\G_A}\setminus\LT{\H}$ lies in $A^\o$.
 Then, $\LT{\G_A}\setminus\LT{\H}$ is antipodal to $B$; see \Cref{rem:replace_hypothesis_amalgam}.
 Similarly, if $\LT{\G_B}\cap \partial B = \LT{\H}$, then $\LT{\G_B}\setminus\LT{\H}$ is antipodal to $A$.
 
 Consequently, when it is known that $\LT{\G_A}\cap \partial A = \LT{\H}=\LT{\G_B}\cap \partial B$, then the second condition (ii) in \Cref{mainthm:amalgam} is automatically satisfied.
\end{remark}

\begin{corollary}\label{cor:malnormal}
Under the hypothesis of \Cref{mainthm:amalgam},
the subgroup $\H$ is weakly malnormal in both $\G_A$ and $\G_B$.
\end{corollary}
\proof Since $\LT{\H}\subset A\cap B$, the interactive pair assumption implies that  for all  
$\alpha\in \G_A\setminus \H$ and $\beta\in \G_B \setminus \H$,  
$$
\alpha(\LT{\H})\cap \LT{\H}= \beta  (\LT{\H})\cap \LT{\H}=\emptyset. 
$$
If, say, $\alpha \H \alpha^{-1}\cap \H$ were infinite, it would contain an infinite order element $\eta\in \H$;\footnote{Every torsion subgroup of a hyperbolic group is finite; see \cite[Ch. 8, Corollary 36]{GdlH90}.} hence, 
$$
\alpha^{-1}(\LT{\< \eta\>})=\LT{\< \alpha^{-1}\eta\alpha\>}
$$
would be nonempty subsets of $\LT{\H}$. That would be a contradiction. \qed 

\begin{corollary}\label{cor:hyp_amalgam}
 Under the hypothesis of \Cref{mainthm:amalgam}, the subgroup $\G$ of $G$ generated by $\G_A$ and $\G_B$ in $G$ is hyperbolic.
\end{corollary}

\begin{proof}
 If $\H = \G_A\cap\G_B$ is quasiconvex in one of $\G_A$ or $\G_B$, then it is quasiconvex in both of them: This follows from the general fact that if $\G$ is a $\tmod$-Anosov subgroup of $G$ and $\H<\G$ is a subgroup, then $\H$ is quasiconvex in $\G$ if and only if $\H$ is a $\tmod$-Anosov subgroup of $G$.
 This general fact is a consequence of the {\em $\tmod$-URU} characterization of $\tmod$-Anosov subgroups; see \cite[Equivalence Theorem 1.1 \& Remark 1.2(i)]{MR3736790}.
 
 Thus, under the hypothesis \Cref{mainthm:amalgam}, $\H$ is quasiconvex in $\G_A$ and $\G_B$. 
 Moreover, by \Cref{prop:faithful_amalgam}, $\<\G_A,\G_B\>$ is naturally isomorphic to $\G_A \star_\H \G_B$.
 Then, together with \Cref{cor:malnormal}, \Cref{thm:Bestvina-Feign}(i) implies that $\<\G_A,\G_B\>$ is hyperbolic.
\end{proof}

\subsection{Alternating sequences}\label{sec:regularity_amalgam}

Recall the notion of alternating sequences from \Cref{def:alternating}.

The main result of this subsection is as follows:

\begin{proposition}\label{lem:shrinks}
 If $(\omega_n)$ is a type A alternating sequence in $\G = \G_A \star_\H \G_B$, then the nested sequence of compact subsets of $\Ft$,
 \[
  \omega_1 B \supset \omega_2 A \supset \omega_3 B \supset \omega_4 A \supset \cdots,
 \]
 converges to a point $\tau \in \Ft$.
 
 Similarly, if $(\omega_n)$ is a type B alternating sequence in $\G$, then the nested sequence of compact subsets of $\Ft$,
 \[
  \omega_1 A \supset \omega_2 B \supset \omega_3 A \supset \omega_4 B \supset \cdots,
 \]
 converges to a point $\tau\in \Ft$.
\end{proposition}

For a (type A or B) alternating sequence $\underline{\omega} = (\omega_n)$, let us denote the point $\tau\in\Ft$ obtained as a limit in the above by $\tau_{\underline{\omega}}$.
As a direct corollary of the above and \Cref{prop:regulairty}, we obtain the following:

\begin{corollary}\label{cor:equivariance_amalgam}
 If $\underline{\omega}=(\omega_n)$ is an alternating sequence in $\G = \G_A \star_\H \G_B$, then, for all $\g\in \G$, 
 \[
 \g\omega_n \tof \g\tau_{\underline{\omega}},\quad \text{as } n\to\infty.
 \]
\end{corollary}

\subsubsection{Regularity}

We begin by showing that alternating sequences are $\tmod$-regular (see \Cref{cor:alt}). Instead of directly proving the $\tmod$-regularity of an alternating sequence, our approach is to study the dynamics of the corresponding inverse sequence and establish its $\tmod$-regularity, which appears to be easier. This leads us to the following definition:

\begin{definition}[Special sequences]
 A sequence $(\tg_n)$ in $\G = \G_A\star_\H\G_B$ is {\em special} if there exist sequences $(\a_n)$ in $\G_A\setminus\H$, $(\b_n)$ in $\G_B\setminus \H$, and an increasing function $F : \N\to\N$ such that
\begin{equation}\label{eqn:special_form}
  \tg_n = \b_{F(n)}\a_{F(n)}\b_{F(n)-1}\a_{F(n)-1}\cdots\b_1\a_1.
\end{equation}
\end{definition}

The advantage of this definition is that subsequences also inherit the property of being special. 
This characteristic will be essentially used in the proof of the following lemma:

\begin{lemma}\label{lem:special1}
 Special sequences in $\G$ are $\tmod$-regular.
\end{lemma} 

\begin{proof}
 Consider any special sequence $(\tg_n)$ in $\G$, and assume that the normal form of $\tg_n$ is given by \eqref{eqn:special_form}.
 Consider any subsequence of $(\tg_n)$, again denoted by $(\tg_n)$.
 
 If the inverse sequence $(\beta_{F(n)}^{-1})$ corresponding to the leftmost letter sequence diverges away from $\H$, then let $\hb_{F(n)} \coloneqq \hb_{F(n)}\pr_\H(\hb_{F(n)}^{-1})$.
 Since $(\hb_{F(n)})$ is $\tmod$-regular,\footnote{This follows by the assumption that $\G_B$ is $\tmod$-Anosov and $(\hb_{F(n)})$ is a divergent sequence in $\G_b$.} after extraction, there exists $\tau_\pm\in\Ft$ so that
 \[
  \hb_{F(n_k)}^{\pm1} \tof \tau_\pm,
 \]
 as $k\to\infty$.
 Note that $\tau_-\not\in \LT{\H}$ (cf. \Cref{lem:proj}(i)).
 Therefore,
 by the antipodality hypothesis (ii) of \Cref{mainthm:amalgam}, $A$ is a compact subset of $C(\tau_-)$.
 Hence, $\hb_{F(n_k)} A \tof \tau_+$ as $k\to\infty$.
 Note that for all $n\in\N$, $\tg_n B \subset \hb_n A$.
 Thus,
 $ \tg_{n_k} B \tof \tau_+$,
  as  $k\to \infty$.
 By \Cref{prop:regulairty}, $(\tg_{n_k})$ is $\tmod$-regular.
 
 Otherwise, after passing to a further subsequence $(\tg_{n_k})$ of $(\tg_n)$ we can (and will) assume that there exists an element $\tilde\beta \in \G_B \setminus \H$ and a sequence $(\eta_k)$ in $\H$ such that
 \begin{equation}\label{eqn:betahat}
  \beta_{F(n_k)} = \tilde\beta \eta_k, \quad \forall n\in\N.
 \end{equation}
  Set $B' = \tilde\beta (A) \subset B^\o$.
  Note that for all $k\in\N$, $\tg_{n_{k+1}}\tg^{-1}_{n_k} B \subset B'$.
  So, by \Cref{lem:mainlemma}, $(\tg_{n_k})$ is $\tmod$-regular.
  
  Therefore, every subsequence of the original sequence contains a $\tmod$-regular subsequence, showing that the original sequence is $\tmod$-regular.
   This concludes the proof.
\end{proof}

For an alternating sequence $(\omega_n)$, applying \Cref{lem:special1} to $(\omega_n^{-1})$, we obtain the following result.

\begin{corollary}\label{cor:alt}
 Alternating sequences in $\G = \G_A \star_\H \G_B$ are $\tmod$-regular.
\end{corollary}

\subsubsection{Proof of \Cref*{lem:shrinks}}

Let us assume that $(\omega_n)$ is of type A; the other possibility can be similarly treated.
 Let $(\b_n)$ denote the rightmost letter sequence corresponding to $(\omega_{2n})$, see \Cref{def:alternating}.
 Consider the special sequence $(\omega_{2n}^{-1})$.
 
 We first show that there exists a sequence $(\eta_n)$ in $\H$ such that $(\ho_{2n}^{-1})$, where $\ho_{2n} = \omega_{2n}\eta_n^{-1}$, has a subsequence that is $\tmod$-regular and $\tmod$-converges to some point $\tau_-$ antipodal to $A$:
 By  \Cref{lem:proj},
 there exist sequences $(\eta_n)$ and $(\hat\eta_n)$ in $\H$ such that both the sequences $(\hb_n)$ and $(\hb^{-1}_n)$ accumulate outside $\geo \H$ in ${\bar\G}_B = \G_B \sqcup \geo\G_B$, where $\hb_n = \hat\eta_n \b_n\eta_n^{-1}$.
 Note that $(\ho_{2n})$, where $\ho_{2n} = \omega_{2n}\eta_n^{-1}$, is still alternating\footnote{Indeed, the corresponding sequences in $\G_A$ and $\G_B$ for $(\ho_n)$ can be taken to be $(\eta_{n-1}\a_n)$ and $(\b_n\eta_n^{-1})$, respectively.} and, hence, by \Cref{cor:alt}, is a $\tmod$-regular sequence. 
 Passing to a subsequence, $(\hb_{n_k})$ is (i) either a  constant sequence, $\hb \in \G_B \setminus \H$, (ii) or 
\begin{equation}\label{eqn:shrinks}
  \hb_{n_k}^{\pm1} \tof \tau_\pm \in \LT{\G_B}\setminus\LT{\H}, 
 \quad \text{as }k\to\infty.
\end{equation}
 If (i) holds, then notice that $\ho^{-1}_{2n_k} B \subset \hb^{-1} A$. 
 In this case, by \Cref{lem:regularity_four}, all the $\tmod$-limit points of $(\ho^{-1}_{2n_k})$ must lie in $\hb^{-1} A \subset B^\o$. By \Cref{rem:replace_hypothesis_amalgam}, $\hb^{-1} A$ is antipodal to $A$.
  If the (ii) holds, then we observe that $\hat{\omega}^{-1}_{2n_k}(B)\subset \hat{\beta}^{-1}_{n_k}(A)$.
  Then, 
 the sequence $(\hat{\beta}^{-1}_{n_k}A)$, and hence $(\hat{\omega}^{-1}_{2n_k}B)$, converges to the $\tmod$-limit point $\tau_-$ (see \eqref{eqn:shrinks}) of the sequence $\hat\beta_{n_k}^{-1}$. By the antipodality assumption (ii)  in the hypothesis of \Cref{mainthm:amalgam}, $\tau_-$ is antipodal to $A$. 
 
 Let $(\eta_n)$ be a sequence in $\H$ as in the preceding paragraph.
 Let $\ho_{2n} = \omega_{2n}\eta_n^{-1}$.
 By \Cref{cor:alt}, $(\ho_{2n})$ is $\tmod$-regular (cf. the preceding paragraph).
 By above, we can (and will) extract a subsequence $(\ho_{2k_n})$ of $(\ho_{2n})$, such that the inverse sequence $(\ho_{2k_n}^{-1})$ $\tmod$-converges to some point $\tau_-$, which is antipodal to $A$.
 After further extraction, we may assume that $(\ho_{2k_n})$ $\tmod$-converges to some point $\tau\in\Ft$.
 Since $A\subset C(\tau_-)$ is compact, the sequence of subsets
 $(\ho_{2k_n} A)$ converges to $\tau$.
 However, for all $n\in\N$, $\ho_{2n} A = \ho_{2n}\eta_n A = \omega_{2n} A$.
 So, $\omega_{2k_n}A \to \tau$.
 
 Thus, the nestedness of the subsets $(\omega_{2n} A)$ implies $\omega_{2n} A \to \tau$. Then, we also obtain $\omega_{2n-1} (B) \to \tau$ by using $\omega_{2n+1} B \subset \omega_n A$.
 This completes the proof.
\qed

\subsection{The boundary map}\label{sec:brdymap}

Recall from \S\ref{sec:boundary} that there is a $\G$-invariant decomposition of 
 $\geo\G$ as the disjoint union $\vi\G \sqcup \vii\G$, where $\vii\G$ admits an equivariant continuous bijection to the Gromov boundary of the Bass-Serre tree $T$ associated with the amalgamated free product $\G = \G_A \star_\H \G_B$.
 
 In this subsection,
we construct a $\G$-equivariant {\em boundary map}
from the Gromov boundary $\geo \G$ of $\G = \G_A \star_\H \G_B$ to $\Ft$,
\begin{equation}\label{eqn:bdmap}
 \xi : \geo\G \to \Ft.
\end{equation}
We define this map $\xi$ separately on $\vi\G$ and $\vii\G$; see \S\ref{sec:def_vi} for the definition of $\xi\vert_{\vi\G}$ and 
\S\ref{sec:xiii} for the definition of $\xi\vert_{\vii\G}$.

In \S \ref{sec:antipodality_amalgam}, we verify that $\xi$ is an antipodal map. Finally, in 
\S\ref{sec:continuity_amalgam},
we show that the map $\xi$ is {dynamics preserving}.

\subsubsection{Definition of the boundary map for type I points}\label{sec:def_vi}

For $\e\in \vi\G$, pick $\g\in\G$ such that $\g^{-1}\e\in \geo\G_A \cup \geo\G_B$. If $\g^{-1}\e\in \geo\G_A$ (resp. $\g^{-1}\e\in \geo\G_B$), then define
\begin{equation}\label{eqn:xii}
 \xi(\e) \coloneqq \g \xi_A (\g^{-1} \e),\quad
  (\text{resp. } \xi(\e) \coloneqq \g \xi_B (\g^{-1} \e)).
\end{equation}

We first show that the map $\xi: \vi\G\to \Ft$ is well-defined.

\begin{lemma}\label{lem:conjugateempty}
 For $\g\in\G$, if $\g(\geo\G_A) \cap \geo\G_A \ne \emptyset$, then $\g\in\G_A$. The same conclusion holds when $A$ is replaced by $B$.
\end{lemma}

\begin{proof}
 Note that $\g(\geo\G_A) = \geo(\g\G_A\g^{-1})$.
 By \eqref{boundary_disjoint}, the nonemptyness of $\g(\geo\G_A) \cap \geo\G_A$ implies that $d_T(\g\G_A,\G_A)\le 1$.
 Thus, $\g\in\G_A$.
\end{proof}

 By \Cref{lem:conjugateempty}, the element $\g$ in the definition of $\xi$ is unique up to the right multiplication by  elements of $\G_A$ (resp. $\G_B$). Since the maps $\xi_A, \xi_B$ are equivariant for $\G_A, \G_B$, respectively, it follows that the definition of $\xi(\e)$ in \eqref{eqn:xii} does not depend on the choice of $\g\in\G$. 
 
 Finally, note that, by definition, we have the following result:

\begin{lemma}\label{lem:equiv_xii}
  The map $\xi: \vi\G\to \Ft$ in \eqref{eqn:xii} is $\G$-equivariant. 
\end{lemma}

 \subsubsection{Definition of the boundary map for type II points}\label{sec:xiii}
 
 For $\e\in\vii\G$, we choose an alternating sequence $(\omega_n)$ given by \Cref{thm:alternating_one} such that
 $
  \omega_n \toC \e.
 $
 If $(\omega_n)$ is of type A, resp. type B, then define 
\begin{equation}\label{eqn:xiii}
   \xi(\e) \coloneqq \bigcap_{n\in\N} \omega_{2n} A, \quad
  \text{resp. } \xi(\e) \coloneqq \bigcap_{n\in\N} \omega_{2n} B.
\end{equation}
 (cf. \Cref{lem:shrinks}).
 
 A consequence of the following result is that $\xi$ is well-defined on $\vii\G$:
 
\begin{lemma}\label{lem:cont_ii}
 For $\e\in\vii\G$ and a sequence $(\g_n)$ in $\G$, if $\g_n\toC \e$, then
 $
  \g_n \tof \xi(\e).
 $
\end{lemma}

\begin{proof}
 Let $(\omega_n)$ be an alternating sequence as above, which we suppose to be of type A (the type B case can be dealt with similarly).
 Then, by \Cref{thm:alternating_one}, there exists a function $F:\N\to\N$ diverging to infinity and $n_0\in\N$ such that for all $n\ge n_0$, we may (and will) choose a normal form of $\g_n$ containing $\omega_{F(n)}$ as a left subword of that form.
 Furthermore, we may also assume that the function $F$ takes only even values.
 
 We split $(\g_n)_{n\ge n_0}$ into two subsequences: The first subsequence $(\g_{k_n})$ contains all elements of $(\g_n)_{n\ge n_0}$
 with the rightmost letter contained in $\G_A$, and the complementary subsequence $(\g_{l_n})$ includes all elements of $(\g_n)_{n\ge n_0}$
 with rightmost letter contained in $\G_B$.
 Notice that for all $n\in\N$,
 \[
  \g_{k_n} B \subset \omega_{F(k_n)} A \quad \text{and} \quad
  \g_{l_n} A \subset \omega_{F(l_n)} A.
 \]
 Since $\omega_{2n} A \to \xi(\e)$, we obtain that both sequences $(\g_{k_n} B)$ and $(\g_{l_n} A)$ of subsets of $\Ft$ converge to $\xi(\e)$.
 Therefore, by \Cref{prop:regulairty}, $\g_{k_n} \tof \xi(\e)$ and $\g_{l_n} \tof \xi(\e)$, yielding the conclusion.
\end{proof}

Together with \Cref{cor:equivariance_amalgam}, we obtain:

\begin{corollary}\label{cor:equiv_xiii}
 The map in \eqref{eqn:xiii} is $\G$-equivariant.
\end{corollary}

\subsubsection{The boundary map preserves convergence dynamics}\label{sec:continuity_amalgam}

 The following result is an analog of \Cref{lem:cont_ii} for type I boundary points.
 
\begin{lemma}\label{lem:cont_i}
Let $\e\in\vi\G$ and let $(\g_n)$ be a sequence in $\G$.
If $\g_n\toC \e$, then $\g_n\tof \xi(\e)$.
\end{lemma}

\begin{proof}
 Using the equivariance of the $\G$-action, it will be enough to prove the result when $\e\in\geo\G_A\cup\geo\G_B$. We argue by contradiction: 
 
If the assertion is false, then there exists $\e\in\geo\G_A\cup\geo\G_B$ and a sequence $(\g_n)$ in $\G$ such that
\begin{equation}\label{eqn:convergenceam}
   \g_n\toC \e, \text{ but $\xi(\e)$ is not a $\tmod$-accumulation point of }(\g_n).
\end{equation}
By \Cref{prop:good_form}, let us choose $D$-normal forms for each element of $(\g_n)$, for some $D\ge 0$.
After extraction, we may assume that the leftmost and rightmost letters of those forms come from the same group, say $\G_A$ and $\G_*$, respectively, where $*$ is either $A$ or $B$; let $(\a_n)$ be the sequence of those leftmost letters.
We also assume that $* = A$; the other choice can be similarly analyzed. 

Clearly, $(\a_n^{-1})$ cannot diverge away from $\H$: Otherwise, since $(\a_n)$ fellow travels $(\g_n)$, $\a_n\toC \e$ and, if $(\a_n^{-1})$ diverges away from $\H$, then
\[
 \g_n B \subset \a_n B \to \tau,
\]
where $\tau$ must be $\xi(\e)$. But, in this case, \Cref{prop:regulairty} shows that $\g_n\tof \tau$, a disagreement with \eqref{eqn:convergenceam}.

Therefore, after passing to a subsequence, we may assume that the elements of $(\g_n)$ have normal forms with a common leftmost letter $\alpha_1$.
Repeating the same argument to the sequence $(\alpha_1^{-1}\g_n)$ yields another subsequence whose elements have normal forms with a common leftmost letter $\beta_1$.
Thus, the original sequence $(\g_n)$ has a subsequence whose elements have normal forms with two leftmost common letters.
Proceeding inductively, for every $l\in\N$, we can find a subsequence $(\g_n')$ of the original sequence  $(\g_n)$ such that the elements of $(\g_n')$ have normal forms with a common leftmost subword of length at least $l$.
For $l = 3$, \Cref{prop:reln} shows that $(\g_n')$ has bounded nearest-point projections to $\G_A$ and $\G_B$.
Thus, $(\g_n')$ cannot have any accumulation  points in the boundary of $\G_A$ and $\G_B$.
This contradicts our initial assumption that $\e\in\geo\G_A\cup\geo\G_B$.
\end{proof}

Combining \Cref{lem:cont_ii,lem:cont_i}, we obtain the following:

\begin{corollary}\label{cor:dynamics_preserving_amalgam}
 The map $\xi:\geo\G\to \Ft$  preserves convergence dynamics.
\end{corollary}

It follows that $\G$ is $\tmod$-regular; cf. \Cref{rem:asmp}.

\subsubsection{The boundary map is antipodal}\label{sec:antipodality_amalgam}

\begin{proposition}\label{prop:antipodal_amalgam}
  The map $\xi : \geo\G \to \Ft$ in \eqref{eqn:bdmap} (obtained by combining \eqref{eqn:xii} and \eqref{eqn:xiii}) is {\em antipodal}: That is
for every pair of distinct points $\e_\pm \in\geo\G$, the points $\tau_\pm \coloneqq \xi(\e_\pm)\in\Ft$ are antipodal to each other.
\end{proposition}

We recall that {\em the action $G\acts G/P$ preserves antipodality}: That is, if $\hat\tau_\pm \in G/P$ is an antipodal pair, then, for all $g\in G$, $g\hat\tau_\pm$ is also an antipodal pair.

\begin{lemma}\label{lem:antipodality_typei}
 Let $v,w$ be any vertices in the Bass-Serre tree $T$ such that $d_T(v,w) \ge 2$.
 Then, $\xi(\geo\G_v) = \LT{\G_v}$ is antipodal to $\xi(\geo\G_{w}) = \LT{\G_{w}}$.
\end{lemma}

\begin{proof}
 By by equivariance, it is enough to assume that $w = \G_A$ or $w = \G_B$.
 We assume the former, i.e., $w = \G_A$; the possibility of $w = \G_B$ can be analyzed similarly.
 Since $d_T(v,w) \ge 2$, we have that $\geo\G_A \cap \geo\G_v = \emptyset$ (see \eqref{boundary_disjoint}).
 
 Suppose first that $v$ is a coset of $\G_B$; so let $\g\in\G$ be any element such that $\g \G_B = v$.
 It is easy to see that such an element $\g$ must satisfy $\rl(\g) \ge 2$.
 We may also choose $\g$ so that it has a normal form whose rightmost letter lies in $\G_A\setminus \H$.
 If the leftmost letter $\alpha$ of that normal form also lies in $\G_A\setminus \H$, then $\rl(\g) \ge 3$, and $\LT{\G_{\alpha^{-1}v}} \subset \b\alpha' A \subset B^\o$, where $\beta$ and $\a'$ are the second and third letters from the left in that normal form of $\g$, respectively.
 Since $A$ and $B^\o$ are antipodal to each other (see \Cref{rem:replace_hypothesis_amalgam}), and $\LT{\G_A} \subset A$, we have that $\LT{\G_A}$ and $\LT{\G_{\alpha^{-1}v}}$ are antipodal to each other, hence so is the pair $\LT{\G_A} = \alpha\LT{\G_A}$ and $\LT{\G_{v}} =\a\LT{\G_{\alpha^{-1}v}}$. 
 If the leftmost letter  of the normal form of $\g$ is some element $\b\in \G_B\setminus \H$, then, by a similar argument, it follows that $\LT{\G_v} \subset B^\o$, which is antipodal to $\LT{\G_A} \subset A$.
 
 Suppose now that $v$ is a coset of $\G_A$.
 In this case, we can choose an element $\g\in\G$ such that $\g \G_A = v$, $\rl(\g)\ge 1$, and the rightmost letter of a normal form of $\g$ is an element of $\G_B\setminus\H$.
 Adapting a similar argument as above, the result follows in this case as well.
\end{proof}

\begin{proof}[Proof of \Cref*{prop:antipodal_amalgam}]
 Combining the following cases, it would follow that the map $\xi : \geo\G \to \Ft$ is antipodal.
 Recall that $\xi$ is $\G$-equivariant (by \Cref{lem:equiv_xii} and \Cref{cor:equiv_xiii}).

\setcounter{case}{0}
\begin{case}
  {\em Suppose that both points $\e_\pm\in\geo\G$ are of type I.}
Since $\xi$ is $\G$-equivariant, it is enough to assume that $\e_-\in \geo\G_A \cup \geo\G_B$.
Let us also assume that $\e_-\in\geo\G_A$; the case $\e_-\in\geo\G_B$ can be treated similarly.

 If $\e_+ \in \geo\G_A \cup (\G_A(\geo\G_B))$, then finding a suitable element $\alpha\in\G_A$, we obtain that $\a\e_+\in \geo\G_A \cup \geo\G_B$.
Since, by the hypothesis of \Cref{mainthm:amalgam},
$\xi(\geo\G_A \cup \geo\G_B) \subset \Ft$ is an antipodal subset, 
$\xi(\alpha\e_\pm)$ is an antipodal pair and, hence, so is $\xi(\e_\pm)$.

 If $\e_+ \not\in \geo\G_A \cup (\G_A(\geo\G_B))$, then 
 $\e_+$ lies in the boundary of a vertex group $\G_v$ of the Bass-Serre tree $T$ such that $d_T(\G_A,v) \ge 2$.
 By \Cref{lem:antipodality_typei}, $\xi(\e_\pm)$ are antipodal.
\end{case}

\begin{case}
 {\em Suppose that both points $\e_\pm\in\geo\G$ are of type II.}
 Consider a pair of alternating sequences $(\omega^\pm_n)$ converging (in $\bG$) to $\e_\pm$ (see \Cref{prop:Dalternating}).
 
\begin{lemma}\label{lem:two:antipodality_typei}
 There exists $n\in\N$ such that $\omega^+_n\not\in \omega_n^- \H$.
\end{lemma}
\begin{proof}
 If this is false, then for all $n\in\N$, $\omega^+_n\in \omega_n^- \H$.
 Consider the sequence $(\ho_n^{-1})$, where $\ho_n \coloneqq \omega_n^-\pr_\H((\omega_n^-)^{-1})$.
 By \Cref{lem:proj}, $(\ho_n^{-1})$ has no accumulation points in $\geo\H$.
 Moreover, $(\omega_n^-)$ diverges away from $\H$ (cf. \Cref{prop:reln}).
 Thus, $(\ho_n(\geo\H))$ converges to $\e_-$, see \Cref{lem:proj}(ii).
 It follows that the sequence of uniformly quasiconvex subsets $(\ho_n(\H))$ of $\G$ converges to $\e_-$.
 Since we have assumed that $\omega^+_n\in \omega_n^- \H = \ho_n \H$, we obtain that $\omega^+_n \toC \e_-$, which shows that $\e_+ = \e_-$. This is a contradiction. 
\end{proof}

Let $n_0\in\N$ be the smallest number such that
 $\omega^+_{n_0}\not\in \omega_{n_0}^- \H$.
 Consider the alternating sequences $((\omega^+_{n_0})^{-1}\omega_n^+)$ and $((\omega^+_{n_0})^{-1}\omega_n^-)$ converging to $(\omega^+_{n_0})^{-1}\e_+$ and $(\omega^+_{n_0})^{-1}\e_-$, respectively.
 By choice of $n_0$, it is evident that the first element of those sequences lie in different groups $\G_A$ and $\G_B$.
 Therefore, one of the points $\xi((\omega^+_{n_0})^{-1}\e_\pm)$ lies in the interior of $A$ while the other one lies in the interior of $B$ (cf. \Cref{lem:shrinks}) and thus they are antipodal.
 Consequently, $\xi(\e_\pm)$ are antipodal.
\end{case}

\begin{case}
 {\em Suppose that $\e_-\in \geo\G$ is of type I and $\e_+\in \geo\G$ is of type II.}
 By the same argument as in the third case in \cite[\S 4.2]{DK22}, it follows that $\xi(\e_\pm)$ are antipodal.\qedhere
\end{case}
\end{proof}

\subsection{Proof of Theorem \ref*{mainthm:amalgam}}\label{sec:proof_amalgam}

By \Cref{prop:faithful_amalgam}, the subgroup $\G = \< \G_A,\G_B\>$ of $G$ is naturally isomorphic to $\G_A\star_\H \G_B$.
We show that $\G$ is a $\tmod$-Anosov subgroup or, equivalently, a $\tmod$-asymptotically embedded subgroup (see \Cref{def:asymptotically_embedded}) of $G$:

\begin{enumerate}[(i)]
\item By \Cref{cor:hyp_amalgam}, $\G$ is hyperbolic.

\item Finally, the boundary map $\xi : \geo\G \to \Ft$ in \Cref{eqn:bdmap} is
 $\G$-equivariant (by \Cref{lem:equiv_xii} and \Cref{cor:equiv_xiii}), antipodal (by \Cref{prop:antipodal_amalgam}), and preserves convergence dynamics (by \Cref{cor:dynamics_preserving_amalgam}).
\end{enumerate}

This concludes the proof of the \Cref{mainthm:amalgam}.
\qed

\section{Proof of \Cref{mainthm:HNN}}
\label{sec:CT_B}

Throughout this section, we work under the hypothesis of \Cref{mainthm:HNN}.

\begin{assumption}\label{rem:replacei}
 In the proof of \Cref{mainthm:HNN}, we will replace $(A,B_\pm)$ by $(\cl (A^\o), \cl(B^\o_\pm))$,
 which is again an interactive triple for $(\P; \,\H_\pm; \, f)$.
 In doing so, we may replace the condition (i) in the hypothesis of \Cref{mainthm:HNN} by the following stronger one (cf. \Cref{lem:replacei}):
 {\em\begin{enumerate}[(i)]
 \item [(i)'] The pairs of subsets $(A,B_\pm^\o)$, $(A^\o,B_\pm)$ of $\Ft$ are antipodal. 
 Moreover, $B_-$ is antipodal to $B_+$.
 \end{enumerate}}
\end{assumption}

We begin the proof of \Cref{mainthm:HNN} with the following observation:

\begin{lemma}\label{lem:limitsetsHNN}
Under the hypothesis of \Cref{mainthm:HNN},
\begin{enumerate}[(i)]
 \item  $\LT{\<f\>}$ consists of two points, one of them lies in the interior of $B_+$ and the other one lies in the interior of $B_-$,
 and $\< f\>$ is a cyclic $\tmod$-Anosov subgroup of $G$.
 \item $\LT{\P} \subset A$ and $\LT{{\H_\pm}} \subset A\cap B_\pm$.
\end{enumerate}
\end{lemma}

\begin{proof}
 Since, for all $n\in\N$, $f^{n+1}f^{-n}({B_+}) = f({B_+}) \subset B_+^\o$ (by the second condition of \Cref{def:interactive_triple}), applying \Cref{lem:mainlemma}, it follows that $\<f\>$ is $\tmod$-regular.
 Since, for all $n\in\N$, $f^{-n} (f^{-1}B_-) \subset f^{-1}B_- \subset B_-^\o$, all the $\tmod$-limit points of the sequence $(f^{-n})_{n\in\N}$ lie in $B_-^\o$.
 Since $B_+$ is antipodal to $B_-^\o$ (see \Cref{rem:replacei}),
 by \Cref{cor:regulairty}, $(f^nB_+)_{n\in\N}$ shrinks.
 Therefore, since the sequence $(f^n B_+)_{n\in\N}$ is nested, $(f^n B_+)_{n\in\N}$ must converge to some point $\tau_+\in B_+^\o$.
 Similarly, the nested sequence $(f^{-n}B_-)_{n\in\N}$ of compact subsets of $\Ft$ converges to some point $\tau_-\in B_-^\o$.
 In particular, the limit set of $\< f\>$ is $\{\tau_\pm\}$, is antipodal, and has cardinality two.
 That $\< f\>$ is $\tmod$-Anosov follows from \cite[Lemma 5.38]{MR3736790}.
 This proves (i).

 Proof (ii) is similar to that of \Cref{lem:ping-pong}.
 Hence, we omit the details.
\end{proof}

\begin{remark}\label{rem:relaxation_HNN}
 Suppose that it is known that $\LT{\P}\cap \partial A = \LT{\H_+}\cup\LT{\H_-}$.
 Then, by \Cref{lem:limitsetsHNN}(ii), the subset $\LT{\P} \setminus \LT{\H_+}$ lies in $A^\o\cup B_-$ and, hence, it is antipodal to $B_+$; see \Cref{rem:replacei}.
 Similarly, $\LT{\P} \setminus \LT{\H_-}$ is antipodal to $B_-$.
 Thus, in this situation, the second condition (ii) of \Cref{mainthm:HNN} is automatically satisfied.
\end{remark}

\begin{corollary}\label{cor:malnormalHNN}
Under the hypothesis of \Cref{mainthm:HNN},
the subgroups $\H_\pm$ are weakly malnormal in $\P$, and, for all $\p\in\P$, the intersection $\H_-\cap \p\H_+\p^{-1}$ is finite.
\end{corollary}

\begin{proof}
 The proof is similar to the one of \Cref{cor:malnormal}; we omit the details.
\end{proof}

\begin{corollary}\label{cor:hyp_HNN}
 Under the hypothesis of \Cref{mainthm:HNN},
the subgroup $\G$ of $G$ generated by $\P$ and $f$ in $G$ is hyperbolic.
\end{corollary}

\begin{proof}
 Arguing similarly to the first paragraph of the proof of \Cref{cor:hyp_amalgam}, it follows that $\H_\pm$ are both quasiconvex in $\P$.
 Then, the claim follows from \Cref{thm:Bestvina-Feign}(ii), \Cref{prop:faithful_HNN}, and \Cref{cor:malnormalHNN}. Compare this with the second paragraph of the proof of \Cref{cor:hyp_amalgam}.
\end{proof}

For convenience, we introduce the following notation, which is frequently used in this section.

\begin{notation}\label{rem:epsilon}
 If $\epsilon = 1$, then $\H_\epsilon$ will denote $\H_+$ and $B_\epsilon$ will denote $B_+$.
 Similarly, if $\epsilon = -1$, then $\H_\epsilon$ will denote $\H_-$ and $B_\epsilon$ will denote $B_-$.
\end{notation}

\subsection{Alternating sequences}
Recall the notion of alternating sequences from \Cref{defn:alt_HNN}.
The main result of this subsection is as follows:

\begin{proposition}\label{prop:convHNN}
 Let $(\omega_n)$ be an alternating sequence in $\G = \Pf$ in the normal forms given by \eqref{eqn:alt_HNN}:
\begin{equation}\label{eqn:alt_HNNt}
   \omega_n = \p_0 f^{\epsilon_1} \p_1 f^{\epsilon_2}\p_2 \cdots f^{\epsilon_{n-1}}\p_{n-1}f^{\epsilon_n}.
\end{equation}
 Then, the nested sequence of compact subsets $(\omega_n (A\cup B_{\epsilon_{n}}))_n$ of $\Ft$ converges to a point.
\end{proposition}

We remind our reader that we are using the notation introduced in \Cref{rem:epsilon}.

By the above proposition, it follows that an alternating sequence $\underline{\omega} = (\omega_n)$ in $\G$ $\tmod$-converges to the  point 
\begin{equation}\label{eqn:defnconvHNN}
 \tau_{\underline{\omega}} \coloneqq \bigcap_{n\in\N} \omega_n(A\cup B_{\epsilon_n}).
\end{equation}
As a corollary, we obtain:

\begin{corollary}\label{cor:convHNN}
 If $\underline{\omega} = (\omega_n)$ is an alternating sequence in $\G$, then the sequence $(\omega_n A)$ of compact subsets of $\Ft$ converges to the $\tmod$-limit point $\tau_{\underline{\omega}}$ of $(\omega_n)$.
\end{corollary}

However, we remark that the sequence $(\omega_n A)$ is possibly not nested.
For example, consider an alternating sequence $(\omega_n)$ given by $\omega_n = (\eta f)^{n}$ for $n\in\N$, where $\eta\in \H_+$.
Then, for all $m\ge n+2$, we have $\omega_m(A)\cap \omega_n(A) = (\eta f)^{n} ( ((\eta f)^{m-n} A) \cap A) \subset f(\eta  f)^{n} (B^\o \cap A) = \emptyset$.

\begin{corollary}\label{cor:equivariance_HNN}
If $\underline{\omega} = (\omega_n)$ is an alternating sequence in $\G$, then, for all $\g\in\G$, 
 \[
  \g\omega_n \tof \g\tau_{\underline{\omega}},
  \quad\text{as } n\to\infty.
 \]
\end{corollary}

\begin{proof}
By  \Cref{cor:convHNN}, it follows that $(\g\omega_n A)_n$ converges to $ \g\tau_{\underline{\omega}}$. Then, \Cref{prop:regulairty} yields the conclusion.
\end{proof}

\subsubsection{Regularity}\label{sec:specialHNN}

\begin{definition}[Special sequences]\label{def:specialHNN}
 A sequence $(\tg_n)$ in $\G = \Pf$ is called {\em special} if 
there exist sequences $(\epsilon_n)$ in $\{\pm1\}$, $(\p_n)$ in $\P$ satisfying
\[
\epsilon_n=1,\p_n \in \H_+\implies \epsilon_{n+1} =1
\quad\text{and}\quad
\epsilon_n=-1,\p_n \in \H_-\implies \epsilon_{n+1} =-1,
\]
and an element $\p_0\in\P$, and an increasing function $F : \N\to\N$ such that for all $n\in\N$, 
\begin{align}\label{eqn:specialHNN}
  \tg_n = f^{\epsilon_{F(n)}}\p_{F(n)-1}f^{\epsilon_{F(n)-1}} \cdots \p_1f^{\epsilon_1}\p_0.
\end{align}
\end{definition}

Note that special sequences are subsequences of the inverse sequence of some alternating sequence in $\G$. Compare this with \Cref{defn:alt_HNN}.

\begin{lemma}\label{lem:special_HNN}
 Special sequences in $\G$ are $\tmod$-regular.
\end{lemma}

\begin{proof}
Let $(\tg_n)$ be a special sequence as above. 
We will assume that $\p_0 = 1_\P$; such a change is a bounded perturbation of the original sequence. Hence the property of being $\tmod$-regular remains unaffected.
To show that $(\tg_n)$ is $\tmod$-regular, it would be enough to show that every subsequence of $(\tg_n)$ contains a $\tmod$-regular subsequence.

So, consider any subsequence of $(\tg_n)$, again denoted by $(\tg_n)$.
After extraction,\footnote{Note that, by definition, subsequences of special sequences are special.} we may assume that $\epsilon_{F(n)-1}$ are all the same, say 
\[
\epsilon_{F(n)-1}=1, \quad\text{for all }n\in\N.
\]

\setcounter{case}{0}
\begin{case}
{\em Suppose that after passing to a subsequence, it holds for all $n\in\N$ that $\p_{F(n)-1} \in \H_+$.}
Passing to another subsequence, we will also assume that $F(n+1) - F(n) \ge 10$, for all $n\in\N$.
Since $\p_{F(n)-1} \in \H_+$ and $\epsilon_{F(n)-1}=1$, we have 
\[
\epsilon_{F(n)}=1,  \quad\text{for all }n\in\N.
\]
 (see \Cref{def:specialHNN}).
So,
\begin{align}\label{eqn:specialHNN:two}
\begin{split}
   \tg_n 
   &= f\p_{F(n)-1}f\p_{F(n)-2}f^{\epsilon_{F(n)-2}} \cdots \p_1f^{\epsilon_1} \\
   &= ff(f^{-1}\p_{F(n)-1}f\p_{F(n)-2}) f^{\epsilon_{F(n)-2}} \cdots \p_1f^{\epsilon_1}.
\end{split}
\end{align}
Moreover,
\[
\tg_{n+1} \tg_n^{-1} = ff(f^{-1}\p_{F(n)-1}f\p_{F(n)-2}) f^{\epsilon_{F(n)-1}} \cdots f^{\epsilon_{F(n-1)+1}}\p_{F(n-1)}.
\]

Observe that if $\epsilon_{F(n-1)+1} = -1$, then $\p_{F(n-1)}\not\in\H_+$, since we observed in the preceding paragraph that $\epsilon_{F(n-1)} =1$.
Else, $\epsilon_{F(n-1)+1} =1$.
Consequently, in both cases (i.e., $\epsilon_{F(n-1)+1} =1$ or $-1$), it holds that
\[
\tg_{n+1} \tg_n^{-1} (B_+) \subset f(B_+) \subset B^\o_+.
\]
Since the above is true for all $n\in\N$, \Cref{lem:mainlemma} applies to show that 
$(\tg_n)$ is $\tmod$-regular.
\end{case}

\begin{case}
{\em Suppose that, after passing to a subsequence, it holds for all $n\in\N$ that $\p_{F(n)-1} \not\in \H_+$.}
Consider the sequence $(\hg_n)$, where
\[
  \hg_n \coloneqq \p_{F(n)-1}f^{\epsilon_{F(n)-1}} \cdots \p_1f^{\epsilon_1} = f^{-\epsilon_{F(n)}}\tg_n.
\]
We show that $(\hg_n)$ is $\tmod$-regular since this would imply that $(\tg_n)$ is $\tmod$-regular.

For all $n\in\N$,
\begin{equation}\label{eqn:specialHNN:one}
 \hg_n A \subset \p_{F(n)-1}B_{+} \subset A.
\end{equation}
If the sequence $(\p_{F(n)-1}^{-1})$ remains at a bounded distance away from $\H_+$, then
after further extraction, we may assume that $(\p_{F(n)-1})$ lies in a single coset $\p\H_+$, for some $\p\not\in\H_+$.
So, it holds that
\[
 \hg_{n+1}\hg_n^{-1}(A) \subset \p B_+ \subset A^\o,
\]
for all $n\in\N$.
With \Cref{lem:mainlemma}, the above implies
that $(\hg_n)$ is $\tmod$-regular.

Otherwise, after further extraction, $(\p_{F(n)-1}^{-1})$ diverges away from $\H_+$.
Consider the sequence $(\hp_{F(n)-1})$, where $\hp_{F(n)-1} \coloneqq \p_{F(n)-1}\pr_{\H_+}(\p_{F(n)-1}^{-1})$.
By \Cref{lem:proj}, $(\hp_{F(n)-1}^{-1})$ accumulates in $\geo\P\setminus \geo\H_+$.
Thus, all $\tmod$-flag accumulation points of $(\hp_{F(n)-1}^{-1})$ are antipodal to $B_+$ (see the second condition in the hypothesis of Theorem \ref{mainthm:HNN}).
By \Cref{cor:regulairty}, the sequence $(\hp_{F(n)-1} B_+)$ of compact subsets of $\Ft$
 shrinks.
 Since $\hp_{F(n)-1} B_+ = \p_{F(n)-1} B_+$, $(\p_{F(n)-1} B_+)$ shrinks.
Consequently, it follows from \eqref{eqn:specialHNN:one} that $(\hg_n A)$ shrinks.
By \Cref{prop:regulairty_two}, $(\hg_n)$ is $\tmod$-regular.
Hence, $(\tg_n)$ is $\tmod$-regular.\qedhere
\end{case}
\end{proof}

By definition, the inverse sequence corresponding to an alternating sequence is special so that \Cref{lem:special_HNN} directly implies:

\begin{corollary}\label{cor:altregHNN}
 Alternating sequences in $\G = \Pf$ are $\tmod$-regular.
\end{corollary}

\subsubsection{Proof of \Cref*{prop:convHNN}}
It is a straightforward verification that for all $n\in\N$,
 $\p_n (B_{\epsilon_{n+1}}) \subset A \cup B_{\epsilon_n}$, yielding that
\begin{align}\label{eqn:alt_HNNfive}
\begin{split}
 \omega_{n+1} (A\cup B_{\epsilon_{n+1}})  &= \omega_{n}\p_n f^{\epsilon_{n+1}}  (A\cup B_{\epsilon_{n+1}})
 \subset \omega_{n}\p_n B_{\epsilon_{n+1}}
 \subset \omega_{n} (A \cup B_{\epsilon_n}).
\end{split}
\end{align}
Therefore, the sequence $(\omega_{n} (A \cup B_{\epsilon_n}))_{n}$ is nested.

\medskip
Thus, to prove that $(\omega_{n}(A \cup B_{\epsilon_n}))_{n}$ converges to a point, it will be enough to show that this sequence contains a subsequence that shrinks.
This is what we show.

\setcounter{case}{0}
\begin{case}
{\em Suppose that for infinitely many $n \in\N$, $\p_{n-1} \in \H_{\epsilon_{n}}$.}
Let $\PP\subset \N$ denote an infinite subset such that for all $n \in\PP$, $\p_{n-1} \in \H_{\epsilon_{n}}$ and, for all $n\in\PP$, $\epsilon_{n}$ are the same, say $\epsilon_{n} = 1$.
Therefore, for all $n\in\PP$, $\epsilon_{n-1} = 1$ (see \Cref{defn:alt_HNN}).

Let us first show that $(\omega_{n-1} B_+)_{n\in\PP}$ shrinks:
We observe that for $n\in\PP$,
\[
 \omega_{n-1}^{-1} (\p_0 A) \subset B_-,
\]
which shows that all the $\tmod$-limit points of $(\omega_{n-1}^{-1})_{n\in\PP}$ lie in $B_-$ (see \Cref{lem:regularity_four}).
By \Cref{cor:altregHNN}, we know that $(\omega_{n-1})_{n\in\PP}$ is $\tmod$-regular.
Since $B_+$ is antipodal to $B_-$, by \Cref{prop:regulairty_two}
$(\omega_{n-1} B_+)_{n\in\PP}$ shrinks.

Since, for all $n\in\PP$, $\epsilon_{n} = 1$,  we get that $B_{\epsilon_{n}} = B_+$.
Moreover, since $\p_{n-1}\in\H_+$, 
\[
 \omega_{n} (A \cup B_{\epsilon_n}) = \omega_{n-1}\p_{n-1} f  (A \cup B_+)
 \subset \omega_{n-1}\p_{n-1} B_+
 = \omega_{n-1} B_+.
\]
Therefore, by the previous paragraph, $(\omega_{n} A)_{n\in\PP}$ shrinks.
\end{case}

Therefore, we may assume now the complementary case to the above one, which is that for at most finitely many $n\in\N$, $\p_{n-1}\in\H_{\epsilon_{n}}$.
In fact, it would be enough to assume:

\begin{case}
{\em For all $n\in\N$, $\p_{n-1} \not\in \H_{\epsilon_{n}}$.}
Suppose that for infinitely many $n\in\N$, $\epsilon_{n} = 1$ (the case $\epsilon_{n}=-1$ can be dealt with in a similar way).
Let $\PP\subset\N$ be a subset such that for all distinct $n,n'\in \PP$, $|n - n'|\ge 10$, and $\epsilon_{n} =1$, for all $n\in \PP$.
For $n\in \PP$, let us consider the normal form of $\omega_{n+1}$ given by
\begin{equation}\label{eqn:alt_HNNeight}
 \omega_{n+1} = \p_0 f^{\epsilon_1} \p_1 \cdots f^{\epsilon_{n-2}}\p_{n-2}  f^{\epsilon_{n-1}}\hp_{n-1}f\hp_{n}f^{\epsilon_{n+1}},
\end{equation}
where 
\begin{equation}\label{eqn:alt_HNNtwo}
 \hp_{n-1} = \p_{n-1}\pr_{\H_+}(\p_{n-1}^{-1})
 \quad\text{and}\quad
 \hp_{n} = f^{-1}(\pr_{\H_+}(\p_{n-1}^{-1}))^{-1}f \p_n.
\end{equation}
We observe that since $\p_{n-1}\not\in\H_+$, $\hp_{n-1}$ is also not an element of $\H_+$, for all $n\in\PP$.

For $n\in \PP$, let 
\begin{equation}\label{eqn:alt_HNNsix}
 \ho_{n} \coloneqq \p_0 f^{\epsilon_1} \p_1 \cdots f^{\epsilon_{n-2}}\p_{n-2}  f^{\epsilon_{n-1}}\hp_{n-1}f = \omega_{n+1} (\hp_{n} f^{\epsilon_{n+1}})^{-1}.
\end{equation}
Since $(\ho_{n})_{n\in\PP}$ is a subsequence of an alternating sequence, by \Cref{cor:altregHNN}, it is $\tmod$-regular.
One directly checks (cf. \eqref{eqn:alt_HNNfive}),
\begin{equation}\label{eqn:alt_HNNfour}
 \omega_{n+1} (A \cup B_{\epsilon_{n+1}}) \subset \ho_{n} A , \quad\text{for all } n\in\PP.
\end{equation}
 After extraction, we may assume that for all $n\in\PP$,
 $\epsilon_{n-1}$ is constant, $\epsilon = \pm1$.
 After another extraction, we also assume that
 $(\hp_{n-1}^{-1})_{n\in\PP}$ either (i) diverges away 
from $\H_{-\epsilon}$, or (ii) remains in a fixed coset $\hp\H_{-\epsilon}$, for some $\hp\in\P$.
 So, for $n\in\PP$,
 \[
  f\ho_n^{-1}(\p_0 A) = \hp_{n-1}^{-1}f^{-\epsilon}\p_{n-2}^{-1}\cdots f^{-\epsilon_1}\p_0^{-1}(\p_0 A).
 \]

In the first case (i), since we are assuming that $(\hp_{n-1})_{n\in\PP}$ diverges away from $\H_{-\epsilon}$, by \Cref{lem:proj_one}, it follows that $(\tp_{n-1})_{n\in\PP}$ is $\tmod$-regular and its $\tmod$-limit points lie only in $\LT{\P}\setminus \LT{\H_{-\epsilon}}$, where $\tp_{n-1} \coloneqq \pr_{\H_{-\epsilon}}(\hp_{n-1})^{-1} \hp_{n-1}$.
Since $\LT{\P}\setminus \LT{\H_{-\epsilon}}$ is antipodal to $B_{-\epsilon}$, by \Cref{cor:regulairty}, $(\tp_{n-1}^{-1} B_{-\epsilon})_{n\in\PP} = (\hp_{n-1}^{-1} B_{-\epsilon})_{n\in\PP}$ shrinks.
So, $(f\ho_{n-1}^{-1}(\p_0 A))_{n\in\PP}$ shrinks as well.
By definition of $\hp_{n-1}$ in \eqref{eqn:alt_HNNtwo}, all the $\tmod$-limit points of $(\hp_{n-1}^{-1})_{n\in\PP}$ lie in $\LT{\P}\setminus \LT{\H_+}$.
Thus, after further extraction, we may assume that
\[
 \hp_{n-1}^{-1} \tof \tau_- \in \LT{\P}\setminus \LT{\H_+},
 \quad\text{as $n\to\infty$ in } \PP.
\]
Therefore, it holds that $f\ho_{n-1}^{-1} (\p_0 A)\to \tau_-$, as ${n\to\infty}$ in $\PP$.
Thus, by \Cref{prop:regulairty} the sequence $(f\ho_{n-1}^{-1})_{n\in\PP}$ $\tmod$-flag converges to $\tau_-$.
Since $\tau_-$ is antipodal to $B_+$, by \Cref{cor:regulairty}, $(\ho_{n-1}f^{-1}B_+)_{n\in\PP}$ shrinks.
Since $fA\subset B_+$, it follows that $(\ho_{n-1}A)_{n\in\PP}$ shrinks.
Thus, by \eqref{eqn:alt_HNNfour}, $(\omega_{n+1} (A \cup B_{\epsilon_{n+1}}))_{n\in\PP}$ shrinks.

In the second case (ii), suppose first that $\epsilon = -1$. Therefore, by the assumption of Case 2 and \eqref{eqn:alt_HNNtwo}, we have that $\hp\not\in\H_{+}$.
Thus, 
\[
f\ho_n^{-1}(\p_0 A) \subset \hp_{n-1}^{-1} B_{+} \subset \hp B_{+} \subset A^\o,
\] for all $n\in \PP$.
So, in this case,  $(f\ho_{n-1}^{-1})_{n\in\PP}$ has no $\tmod$-limit points in $B_+$, since, by above, all of them lie in the interior of $A$ (cf. \Cref{lem:regularity_four}).
Therefore, since $(\ho_{n-1}f^{-1})_{n\in\PP}$ is $\tmod$-regular,\footnote{This follows by the observation above that $(\ho_{n-1})_{n\in\PP}$ is $\tmod$-regular.} by \Cref{cor:regulairty}, $(\ho_{n-1} f^{-1}B_+)_{n\in\PP}$ shrinks.
Thus, by \eqref{eqn:alt_HNNfour}, $(\omega_{n+1} (A \cup B_{\epsilon_{n+1}}))_{n\in\PP}$ shrinks.

Still assuming (ii), suppose now that $\epsilon = 1$.
If $\hp\not\in\H_-$, then proceeding as in the previous paragraph, it follows that $(\omega_{n+1} (A \cup B_{\epsilon_{n+1}}))_{n\in\PP}$ shrinks.
Else, we must have $\hp_{n-1}\in\H_-$, for all $n\in\PP$.
Observing a different normal form of $\ho_{n}$,
\begin{equation*}
 \ho_{n} \coloneqq \p_0 f^{\epsilon_1} \p_1 \cdots f^{\epsilon_{n-2}}(\underbrace{\p_{n-2}  
 f\hp_{n-1}f^{-1}}_{\in\P})f f.
\end{equation*}
it follows by Case 1 that $(\ho_n A)$ shrinks.
Thus, by \eqref{eqn:alt_HNNfour}, $(\omega_{n+1} (A \cup B_{\epsilon_{n+1}}))_{n\in\PP}$ shrinks.\qedhere
\end{case}

\subsection{The boundary map}\label{sec:brdy map:HNN}

We construct a $\G$-equivariant map from the Gromov boundary of $\G = \Pf$ to the flag manifold $\Ft$:
\begin{equation}\label{eqn:bdmapHNN}
  \xi : \geo \G \to \Ft.
\end{equation}
Recall that $\geo\G$ decomposes into $\vi\G \sqcup \vii\G$,
where $\vi \G = \G \cdot (\geo\P )$.
As in the case of amalgamated free products (\S\ref{sec:brdymap}),
we define $\xi$ separately on $\vi\G$ and $\vii\G$;
see \S\ref{sec:def_vi:HNN} for the definition of $\xi\vert_{\vi\G}$ and 
\S\ref{sec:def_vii:HNN} for the definition of $\xi\vert_{\vii\G}$.

\subsubsection{Definition of the boundary map for type I points}\label{sec:def_vi:HNN}
For every point $\e\in\vi\G$, we may pick some element $\g\in\G$ such that
$\g^{-1}\e \in \geo\P$.
Since $\P$ is $\tmod$-Anosov, we have a $\P$-equivariant boundary embedding $\xi: \geo\P \to \LT{\P} \subset \Ft$.
We define
\begin{equation}\label{eqn:def_vi:HNN}
 \xi(\e) \coloneqq \g \xi(\g^{-1}\e).
\end{equation}

We check that $\xi(\e)$ is well-defined, i.e., does not depend on the choice of $\g\in\G$ in \eqref{eqn:def_vi:HNN}:
If $\g_1\in\G$ is any other element such that $\g_1^{-1}\e \in \geo\P$,
then the intersection $(\g^{-1}\g_1)\geo\P \cap \geo\P$ is nonempty since $\g^{-1}\e$ is a common point.
Thus, in the Bass-Serre tree $T$, $\P$ and $(\g^{-1}\g_1) \P$ are equal or adjacent vertices, showing that $\rl(\g^{-1}\g_1)\le 1$ (see \Cref{lem:relnHNN}).
If $\rl(\g^{-1}\g_1)=0$, then $\g_1 \in \g\P$, and, in this case, the well-definedness of \eqref{eqn:def_vi:HNN} follows by $\P$-equivariance of $\xi: \geo\P \to \Ft$.
So, let us assume that $\rl(\g^{-1}\g_1)= 1$.
In this case, $\g^{-1}\g_1 = \p_0f^{\epsilon}\p_1$, where $\p_0,\p_1\in \P$ and $|\epsilon|=1$.
Let us also assume that $\epsilon=1$, since the case $\epsilon=-1$ is similar.
So, $\g_1 = \g \p_0f\p_1$ and $\g^{-1}\e\in \p_0f\p_1(\geo\P)\cap \geo\P = \p_0(\geo\H_+)$.
Hence, 
\begin{equation}\label{eqn:one:def_vi:HNN}
 \p_0^{-1}\g^{-1}\e\in\geo\H_+.
\end{equation}
Since $f\in G$ conjugates $\H_-$ and $\H_+$, i.e., $f\H_-f^{-1} = \H_+$,
for all $\e'\in\geo\H_+$,
\begin{equation}\label{eqn:two:def_vi:HNN}
  \xi\vert_{\geo\H_-}(f\e') = f\xi\vert_{\geo\H_+}(\e').
\end{equation}
Thus,
\begin{align*}
  \g_1\xi(\g_1^{-1}\e) &= \g_1\xi(\p_1^{-1}f^{-1}\p_0^{-1}\g^{-1}\e)\\
 &= \g_1\p_1^{-1}\xi(f^{-1}\p_0^{-1}\g^{-1}\e)\\
 &= \g_1\p_1^{-1}f^{-1}\xi(\p_0^{-1}\g^{-1}\e)\\
 &= \g_1\p_1^{-1}f^{-1}\p_0^{-1}\xi(\g^{-1}\e) = \g\xi(\g^{-1}\e),
\end{align*}
where the second equality is valid because $(f^{-1}\p_0^{-1}\g^{-1}\e)\in\geo\P$,
the third equality is verified by \eqref{eqn:one:def_vi:HNN} and \eqref{eqn:two:def_vi:HNN},
and the fourth equality is valid because $\g^{-1}\e\in\geo\P$.

The following result is immediate from the definition of $\xi$ above:

\begin{lemma}\label{cor:equiv_xiiHNN}
 The map $\xi: \vi\G \to\Ft$ defined by \eqref{eqn:def_vi:HNN} is $\G$-equivariant.
\end{lemma}

\subsubsection{Definition of the boundary map for type II points}\label{sec:def_vii:HNN}
For $\e\in\vii\G$, consider an alternating sequence (see \Cref{defn:alt_HNN})
$\omega_n \toC \e$ given by \Cref{thm:alternating_one}.
Define
\begin{equation}\label{eqn:def_vii:HNN}
 \xi(\e) \coloneqq \bigcap_{n\in\N} \omega_n(A\cup B_{\epsilon_n}),
\end{equation}
see \Cref{prop:convHNN}.
The following result shows that $\xi(\e)$ is well-defined.

\begin{lemma}\label{prop:convviiHNN}
 For $\e\in\vii\G$ and for any sequence $(\g_n)$ in $\G$, if $\g_n \toC \e$, then $\g_n \tof \xi(\e)$.
\end{lemma}

\begin{proof}
 The proof is similar to \Cref{lem:cont_ii}. We omit the details.
\end{proof}

By \Cref{cor:equivariance_HNN}, we have:

\begin{corollary}\label{cor:equiv_xiiiHNN}
 The map $\xi: \vii\G \to\Ft$ defined by \eqref{eqn:def_vii:HNN} is $\G$-equivariant.
\end{corollary}

\subsubsection{The boundary map preserves convergence dynamics}\label{sec:continuity_HNN}

The following result is reminiscent of \Cref{prop:convviiHNN}.

\begin{lemma}\label{prop:convviHNN}
Let $\e\in\vi\G$ and let $(\g_n)$ be a sequence in $\G$.
If $\g_n\toC \e$, then $\g_n\tof \xi(\e)$.
\end{lemma}

\begin{proof}
Using the action of $\G$ on $\vi\G$, it will be enough to prove the proposition for $\e\in \geo\P$.
We argue by contradiction:

If the result is false, then there exists a sequence $(\g_n)$ in $\G$ and $\e\in\geo\P$  such that 
\begin{equation}\label{eqn:convergenceHNN}
   \g_n\toC \e,\text{ but $\xi(\e)$ is not a $\tmod$-accumulation point of }(\g_n).
\end{equation}
Equip each $\g_n$ with a $D$-normal form (see \Cref{prop:good_form}),
\begin{equation}\label{eqn:convergenceHNNtwo}
  \g_n = \p_{n,0} f^{\epsilon_{n,1}}\p_{n,1}\cdots f^{\epsilon_{n,l_n}}\p_{n,l_n}.
\end{equation}
Passing to a subsequence, we may assume that $\epsilon_{n,1}$ are the same, $\epsilon$, for all $n$.
In this situation, $(\p_{n,0}^{-1})$ cannot diverge away from $\H_\epsilon$: 
Otherwise, after extraction, $\p_{n,0}\toC \e$ and
\[
 \g_n B_{\epsilon_{n,l_n}}\subset\p_{n,0} B_\epsilon \to \xi(\e),
\]
showing that $\g_n\tof \xi(\e)$.
However, since $(\p_{n,0})$ fellow-travels $\g_n$, $\g_n\toC\e$.
Thus, $\e=\e$ and, therefore, the assumption \eqref{eqn:convergenceHNN} is violated.

So, after extraction, $\p_{n,0}$ are all the same, $\p_0$.
We may now repeat the same procedure to the sequence $(f^{-1}\p_{0}^{-1}\g_n)_n$ to
show that after another extraction $\p_{n,1}$ are all the same.
We can continue doing this procedure an arbitrary number of times: Thus, for all $l\in\N$, there exists a subsequence $(\g_n')$ of $(\g_n)$ such that 
 suitable normal forms of the elements of $(\g_n')$ share at least $l$ common leftmost letters.
For $l=3$, 
 \Cref{prop:reln} shows that $(\g'_n)$ has no accumulation points in the boundary of $\P$. This is a contradiction with the  assumption that $\g_n\toC \e\in\geo\P$.
\end{proof}

Combining \Cref{prop:convviiHNN,prop:convviHNN}, we obtain:

\begin{corollary}\label{cor:dynamics_preserving_HNN}
 The map $\xi:\geo\G\to \Ft$  preserves convergence dynamics.
\end{corollary}

In particular, $\G$ is $\tmod$-regular; see \Cref{rem:asmp}.

\subsubsection{The boundary map is antipodal}\label{sec:antipodality_HNN}

\begin{proposition}\label{prop:antipodal_HNN}
  The map $\xi : \geo\G \to \Ft$ in \eqref{eqn:bdmapHNN} (obtained by combining \eqref{eqn:def_vi:HNN} and \eqref{eqn:def_vii:HNN}) is {\em antipodal}: That is,
for every pair of distinct points $\e_\pm \in\geo\G$, the points $\tau_\pm \coloneqq \xi(\e_\pm)\in\Ft$ are antipodal to each other.
\end{proposition}

\begin{proof} We consider the following cases.
Recall that $\xi$ is $\G$-equivariant (by \Cref{cor:equiv_xiiiHNN} and \Cref{cor:equiv_xiiHNN}).

\setcounter{case}{0}
\begin{case}
{\em Suppose that both $\e_\pm$ are type I points.}
 Using the $\G$-equivariance, we may assume that $\e_-\in\geo\P$. 
 If $\e_+$ is also in $\geo\P$, then $\xi(\e_\pm)$ are antipodal since $\xi:\geo\P\to\Ft$ is an antipodal map.
 Else, $\e_+\not\in\geo\P$ but there exists $\g\in\G$ such that $\rl(\g)\ge 1$ and $\e \coloneqq \g^{-1}\e_+\in\geo\P$.
 Let $\g = \p_{0} f^{\epsilon_{1}}\p_1 \cdots f^{\epsilon_n} \p_n$ be a normal form.
 So, 
 \[
  \p_{0}^{-1}\xi(\e_+) \in \p_{0}^{-1}\g(\LT{\P}) = \p_{0}^{-1}f^{\epsilon_{1}}\p_1 \cdots f^{\epsilon_n} (\LT{\P}).
 \]
 
 If $\rl(\g) = 1$, then $\p_{0}^{-1}\xi(\e_+) \in f^{\epsilon_{1}}(\LT{\P})$.
 Moreover, $\p_{0}^{-1}\e_+ \not\in f^{\epsilon_{1}}\geo{\H_{-\epsilon_1}}$, since we have assumed that $\e_+\not\in\geo\P$.
 Thus, it follows that $f^{-\epsilon_{1}}\p_{0}^{-1}\xi(\e_+) \in \LT{\P}\setminus\LT{\H_{-\epsilon_1}}$.
 Since $B_{-\epsilon_1}$ is antipodal to $\LT{\P}\setminus\LT{\H_{-\epsilon_1}}$,
 $f^{-\epsilon_1}\p_{0}^{-1}\xi(\e_-)$, which is an element of $ B_{-\epsilon_1}$, is antipodal to $\LT{\P}\setminus\LT{\H_{-\epsilon_1}}$.
 Thus, $\xi(\e_+)$ is antipodal to $\xi(\e_-)$.
 
 If $\rl(\g) \ge 2$, and $\p_1\not\in\H_{\epsilon_2}$, then 
 \[
  \p_{0}^{-1}\xi(\e_+) \in \p_{0}^{-1}\g A\subset f^{\epsilon_1}\p_1 B_{\epsilon_2} \subset B_{\epsilon_1}^\o. 
 \]
 If $\p_1\in\H_{\epsilon_2}$,
 then $\epsilon_1 = \epsilon_2$ and, hence,
 \[
  \p_{0}^{-1}\xi(\e_+) \in \p_{0}^{-1}\g A\subset f^{\epsilon_1}\p_1f^{\epsilon_2} (A\cup B_{\epsilon_2}) = f^{\epsilon_1}B_{\epsilon_1}\subset B^\o_{\epsilon_1},
 \]
 In both cases, $\p_{0}^{-1}\xi(\e_+)$ is antipodal to $A$ and
 in particular, to $\LT{\P}$. Thus, $\xi(\e_+)$ is antipodal to $\LT{\P}$ and, in particular, to $\xi(\e_-)$.
 \end{case}

\begin{case}
{\em Suppose that both $\e_\pm$ are type II points.}
 Let $(\omega_n^\pm)$ be alternating sequences such that
$
  \omega^\pm_n \toC \e_\pm
$
(see \Cref{prop:Dalternating}).

\begin{lemma}
 There exists $n\in\N$ such that $\omega^+_n\not\in \omega_n^- \P$.
\end{lemma}

\begin{proof}
 The proof is similar to the one of \Cref{lem:two:antipodality_typei}.
\end{proof}

Let $n_0$ be the smallest natural number such that $\omega^+_{n_0}\not\in \omega_{n_0}^- \P$.
Using an argument similar to the Case 2 in the proof of 
\Cref{prop:antipodal_amalgam}, one can show that $\xi((\omega_{n_0}^{+})^{-1}\e_\pm)$ are antipodal, which is equivalent to $\xi(\e_\pm)$ being antipodal.
\end{case}

\begin{case}
{\em Suppose that $\e_-$ is a type I point, but $\e_+$ is a type II point.}
 By $\G$-equivariance, we may assume that $\e_-\in\geo\P$.
 Let $(\omega_n)$ alternating sequence, \[\omega_n = \p_0 f^{\epsilon_1} \p_1 f^{\epsilon_2}\p_2 \cdots f^{\epsilon_{n-1}}\p_{n-1}f^{\epsilon_n},\] such that
$
  \omega_n \toC \e_+.
$
 Then, $\xi(\e_+) = \lim_{n\to\infty} \omega_n A$, see \eqref{eqn:def_vii:HNN}.
 
 If $\p_1\in\H_{\epsilon_2}$, then $\epsilon_1 = \epsilon_2$, and it follows that 
 \[
  \p_0^{-1}\omega_n A \subset f^{\epsilon_1} B_{\epsilon_1} \subset B_{\epsilon_1}^\o.
 \] 
 Thus, $\p_0^{-1}\xi(\e_+)\in B_{\epsilon_1}^\o$.
 Else, if $\p_1\not\in\H_{\epsilon_2}$, then
 $
  (\p_0 f^{\epsilon_1} \p_1)^{-1}\omega_n A \subset  B_{\epsilon_2},
 $
 also showing that $\p_0^{-1}\xi(\e_+) \in f^{\epsilon_1} \p_1 B_{\epsilon_2}\subset B_{\epsilon_1}^\o $.
 However, $\p_0^{-1}\xi(\e_-) \in A$.
 It follows that $\xi(\e_-)$ and $\xi(\e_+)$ are antipodal.\qedhere
\end{case}
\end{proof}

\subsection{Proof of Theorem \ref*{mainthm:HNN}}

We show that the subgroup $\G = \< \P,f\> < G$  in the conclusion of \Cref{mainthm:HNN},
which is naturally isomorphic to $\Pf$ (by \Cref{prop:faithful_HNN}),
 is a $\tmod$-Anosov subgroup; this is equivalent to showing that $\G$ is a $\tmod$-asymptotically embedded subgroup (see \Cref{def:asymptotically_embedded}) of $G$:

\begin{enumerate}[(i)]
\item That $\G$ is hyperbolic is the content of \Cref{cor:hyp_HNN}.

\item Finally, the boundary map $\xi : \geo\G \to \Ft$ in \Cref{eqn:bdmapHNN} is
 $\G$-equivariant (by \Cref{cor:equiv_xiiHNN} and \Cref{cor:equiv_xiiiHNN}), antipodal (by \Cref{prop:antipodal_HNN}), and preserves convergence dynamics (by \Cref{cor:dynamics_preserving_HNN}).
\end{enumerate}

This concludes the proof of the \Cref{mainthm:HNN}.
\qed

\section{An application to $\Theta$-positive representations of surface groups}\label{sec:positive}
In this section, we discuss a rich class of Anosov representations of surface groups in certain simple Lie groups $G$, called {\em $\Theta$-positive} representations.

Guichard-Wienhard \cite{GWtheta,GWthetaTwo}, generalizing Lustzig's classical notion of total positivity, introduced the notion of {\em $\Theta$-positivity} for certain flag varieties:
Let $G$ be a simple Lie group of noncompact type and with a finite center, let $P_\Theta$ be a parabolic subgroup conjugate to its opposite, and let $P^{\rm opp}_\Theta$ be an opposite parabolic subgroup.
A {\em positive structure} on the pair $(G,P_\Theta)$ is a sharp open sub-semigroup $U_\Theta^>$ of the unipotent radical $U_\Theta$ of $P_\Theta$, invariant under the conjugation action by the identity component of $P_\Theta\cap P^{\rm opp}_\Theta$.
Here, sharp means that if $u,v\in \cl( U_\Theta^>)$ and $uv=1_G$, then $u=v=1_G$.
The subset $D \coloneqq U_\Theta^>\cdot [P^{\rm opp}_\Theta] \subset G/P_\Theta$ is called a {\em diamond} and ${D^{\vee}} \coloneqq {(U_\Theta^>)}^{-1}\cdot [P^{\rm opp}_\Theta] \subset G/P_\Theta$ is called the {\em opposite diamond}.
Crucial to the discussion of $\Theta$-positivity is the fact that the subsets $D$ and ${D^{\vee}}$ are connected components of $C([P_\Theta])\cap C([P^{\rm opp}_\Theta])$, \cite[Theorem 4.7]{GWtheta}.
The sharp semigroup property of  $U_\Theta^>$ implies that $D$ and ${D^{\vee}}$ are antipodal to each other.
See \cite[\S10]{GWthetaTwo} for a detailed discussion of diamonds.

An $n$-tuple $(\sigma_1,\dots,\sigma_n)$ of distinct points in $G/P_\Theta$ is said to be {\em positive} if there exists $g\in {\rm Aut}(G)$, $u_1,\dots,u_{n-2}\in U_\Theta^>$ such that
\[
 g\cdot (\sigma_1,\dots,\sigma_n)
 = ([P_\Theta], \underbrace{u_{n-2}\cdots u_1[P_\Theta^{\rm opp}],\dots,u_1[P_\Theta^{\rm opp}]}_{\in D}, [P_\Theta^{\rm opp}]).
\]

\begin{definition}
Fix a cyclic order in the circle $S^1$.
A map $\xi:S^1 \to G/P_\Theta$ is said to be {\em positive} if $\xi$ sends every cyclically ordered $n$-tuple of distinct points in $S^1$ to a {\em positive} $n$-tuple. 
\end{definition}

Note that positive maps are not assumed to be continuous.

Suppose that $\xi:S^1 \to G/P_\Theta$ is  a positive map. 
Given a cyclically ordered triple $(\e_+,\e,\e_-)$ in $S^1$,
let
$
 D_\e(\e_+,\e_-)
$
denote the connected component of $C(\xi(\e_+))\cap C(\xi(\e_-))$ that contains $\xi(\e)$.
Then, for any other point $\e'\in S^1$ such that $(\e_+,\e',\e_-)$ is cyclically ordered, we have $D_\e(\e_+,\e_-) = D_{\e'}(\e_+,\e_-)$.
Moreover, if $(\e_1,\dots,\e_4)$ is any cyclically ordered $4$-tuple in $S^1$ and $g\in {\rm Aut}(G)$ is any element such that $g\cdot D_{\e_2}(\e_1,\e_3) = D$, then $g\cdot D_{\e_4} (\e_3,\e_1) = D^\vee$;
in particular,
\begin{equation}\label{eqn:antipodal}
 D_{\e_2}(\e_1,\e_3) \text{ is antipodal to }D_{\e_4} (\e_3,\e_1).
\end{equation}
See \cite[Proposition 2.5(3)]{Guichard:2021aa}.

An interesting characteristic property of positive maps $\xi:S^1 \to G/P_\Theta$ is the following nesting property (\cite[Corollary 3.9]{Guichard:2021aa}), which  reflects the fact that to $U^>_\Theta$ is a sharp semigroup: If $(\e_1,\dots,\e_5)$ is any cyclically ordered $5$-tuple in $S^1$, then
\begin{equation}\label{eqn:nested}
 D_{\e_3}(\e_1,\e_5) \supset \cl\left( D_{\e_3}(\e_2,\e_4) \right).
\end{equation}

\begin{definition}Let $\Gamma$ be a surface group. A representation $\rho:\G \to G$ is called {\em $\Theta$-positive} if 
there exists a $\rho$-equivariant positive map $\xi: \geo\G \to G/P_{\Theta}$. 
\end{definition}

Using the Combination Theorems,\footnote{We remark that the Combination Theorems mainly circumvent the discussion of {\em tripod metrics} in the original proof of this result in \cite{Guichard:2021aa}.} \eqref{eqn:antipodal}, and \eqref{eqn:nested}, we give a proof of the following:

\begin{theorem}[Guichard-Labourie-Wienhard, {\cite[Theorem B]{Guichard:2021aa}}]
 Let $S$ be a closed orientable surface of genus $\ge 2$.
 If $\rho:\pi_1(S) \to G$ is a $\Theta$-positive representation, then $\rho$ is a $\Theta$-Anosov representation.
\end{theorem}

\begin{proof}
 We will identify $\G \coloneqq \pi_1(S)$ with a Fuchsian subgroup of $\PSL(2,\R)$ and, thus, we obtain a hyperbolic structure on $S$.
 Let $c$ be a simple closed separating\footnote{Alternatively, one could also choose a non-separating one and work with  HNN extensions.} geodesic in $S$, whose free homotopy class is represented by an element $\eta\in\G$.
 Then, $\G$ can be written as an amalgamated free product (cf. \S\ref{example})
 \[
  \G = \G_A \star_{\H} \G_B,
 \]
 where $\H \coloneqq \<\eta\> \cong\Z$.
 The subgroup $\G_A$ (and, similarly, $\G_B$) of $\PSL(2,\R)$ is {\em Schottky}, i.e., there exist generators $\a_1,\dots,\a_k$ of $\G_A$ and pairwise disjoint closed intervals in $\geo{\mathbb H}^2$,
 \[
  I^+_1, I_1^-, \dots, I_k^+, I_k^-,
 \]
 such that for all $j\in\{1,\dots,k\}$, $\a_j$ maps the exterior of $I_j^-$ onto the interior of $I_j^+$. 
 
\begin{claim}
 $\rho\vert_{\G_A}$ (and, similarly, $\rho\vert_{\G_B}$) is a $\Theta$-Anosov representation.
\end{claim}

Cf. \cite[Theorem 1.3]{Burelle-Treib}.

\begin{proof}[Proof of claim]
 We first show that for all $j\in\{1,\dots,k\}$, $\H_j \coloneqq \<\rho(\a_j)\>$ is $\Theta$-Anosov. 
 Let $\e_j^\pm$ be the attractive/repulsive fixed point of $\a_j$ in $\geo\mathbb{H}^2$, which lies in the interior of  $I_j^\pm$, respectively.
 Let $x^\pm_j,y^\pm_j$ denote the endpoints of $I_j^\pm$; choose these names in a way such that each triple $(x_j^\pm,\e_j^\pm,y_j^\pm)$ is cyclically ordered.
 If follows that the $5$-tuples $(x_j^\pm,\a_j^{\pm1}x_j^\pm,\e_j^\pm,\a_j^{\pm1}y_j^\pm,y_j^\pm)$ are also cyclically ordered.
  
 Let $A_j^\pm \coloneqq \cl(D_{\e_j^\pm}(x_j^\pm,y_j^\pm))$.
 Then, \eqref{eqn:nested} implies that
 \[
  (A_j^\pm)^\o \supset \cl\left(D_{\e_j^\pm}(\a_j^{\pm1}x_j^\pm,\a_j^{\pm1}y_j^\pm)\right) = \rho(\a_j^{\pm1}) A_j^\pm,
 \]
 whereas \eqref{eqn:antipodal} and \eqref{eqn:nested} imply that $A_j^+$ is antipodal to $A_j^-$.
 Then, by the same proof as in \Cref{lem:limitsetsHNN}(i), we see that $\H_j = \<\rho(\a_j)\>$ is $\Theta$-Anosov and $\LT{\H_j} = \{\xi(\e_j^+),\xi(\e_j^-)\}$.
 
 The subsets $A_j\coloneqq A_j^+\cup A_j^-$, for $j=1,\dots,k$ are also pairwise antipodal. Then,
 \[
  \a_j^n\left(\bigcup_{i\ne j}(I_i^+\cup I_i^-) \right) \subset (I_j^+\cup I_j^-)^\o,
 \]
 for all nonzero $n\in\Z$,
 simply translates to (using \eqref{eqn:nested}) the fact that
 \[
  \rho(\a_j^n)\left(\bigcup_{i\ne j}A_i \right) \subset A_j^\o. 
 \]
 Since we already observed in the preceding paragraph that $\H_1,\dots,\H_k$ are $\Theta$-Anosov subgroups of $G$, using the above, the combination theorem \cite[Corollary 6.3]{DK22} from our previous work yields that $\rho: \G_A\to  \<\H_1,\dots,\H_k\>$ is $\Theta$-Anosov.
\end{proof}

Let $I_A$ and $I_B$ denote the closures of the connected components  of the complement of the 2-point limit set $\{ \e_+,\e_-\}$ of $\H$ in $\geo\mathbb{H}^2$.
Pick any interior points $\e_A\in I_A$ and $\e_B\in I_B$.
We chose these names so that $\Lambda(\G_A)\subset I_A$, $\Lambda(\G_B)\subset I_B$, and $(\e_A,\e_+,\e_B,\e_-)$ is cyclically ordered. Let 
\[
 A \coloneqq \cl(D_{\e_A}(\e_-,\e_+)) \quad\text{and}\quad
 B \coloneqq \cl(D_{\e_B}(\e_+,\e_-));
\]
by \eqref{eqn:antipodal}, $A^\o$ and $B^\o$ are antipodal to each other.
Moreover, since, for all $\a\in\G_A\setminus\H$ and $\beta\in\G_B\setminus\H$, 
$$
\a I_B\subset I_A^\o\quad \hbox{and}\quad  \b I_A\subset I_B^\o,$$
it follows that
\[
 \rho(\a) B\subset A^\o \quad\text{and}\quad \rho(\b) A\subset B^\o.
\]
Furthermore, $\rho(\H)$ preserves $\xi(I_A)$ and $\xi(I_B)$, so that $\rho(\H) A = A$ and $\rho(\H) B = B$.
Thus, $(A,B)$ is an interactive pair for $(\rho(\G_A),\rho(\G_B);\,\rho(\H))$.
Finally, using the antipodality property of the $\Theta$-limit sets of $\rho(\G_A)$ and $\rho(\G_B)$ and \Cref{lem:replacei,lem:ping-pong}, one may directly verify the hypothesis (ii) in the statement of \Cref{mainthm:amalgam}; see also \Cref{rem:relaxation}.
Thus, \Cref{mainthm:amalgam} implies that $\rho$ is $\Theta$-Anosov; cf. the examples in \S\ref{example}.
\end{proof}


\begin{thebibliography}{10}

\bibitem{baker2008combination}
Mark Baker and Daryl Cooper.
\newblock A combination theorem for convex hyperbolic manifolds, with
  applications to surfaces in 3-manifolds.
\newblock {\em J. Topol.}, 1(3):603--642, 2008.

\bibitem{MR823981}
Werner Ballmann, Mikhael Gromov, and Viktor Schroeder.
\newblock ``Manifolds of nonpositive curvature,'' volume~61 of {\em Progress
  in Mathematics}.
\newblock Birkh\"{a}user Boston, Inc., Boston, MA, 1985.

\bibitem{Bestvina-Feighn}
Mladen Bestvina and Mark Feighn.
\newblock A combination theorem for negatively curved groups.
\newblock {\em J. Differential Geom.}, 35(1):85--101, 1992.

\bibitem{Burelle-Treib}
Jean-Philippe Burelle and Nicolaus Treib.
\newblock Schottky presentations of positive representations.
\newblock {\em Math. Ann.}, 382(3-4):1705--1744, 2022.

\bibitem{DK22}
Subhadip Dey and Michael Kapovich.
\newblock Klein-{M}askit combination theorem for {A}nosov subgroups: free
  products.
\newblock {\em Math. Z.}, 305(2):Paper No. 35, 25, 2023.

\bibitem{MR4002289}
Subhadip Dey, Michael Kapovich, and Bernhard Leeb.
\newblock A combination theorem for {A}nosov subgroups.
\newblock {\em Math. Z.}, 293(1-2):551--578, 2019.

\bibitem{Drutu-Kapovich}
Cornelia Dru\c{t}u and Michael Kapovich.
\newblock ``Geometric group theory,'' volume~63 of {\em American Mathematical
  Society Colloquium Publications}.
\newblock American Mathematical Society, Providence, RI, 2018.
\newblock With an appendix by Bogdan Nica.

\bibitem{eberlein}
Patrick~B. Eberlein.
\newblock ``Geometry of nonpositively curved manifolds.''
\newblock Chicago Lectures in Mathematics. University of Chicago Press,
  Chicago, IL, 1996.

\bibitem{MR2233852}
Vladimir Fock and Alexander Goncharov.
\newblock Moduli spaces of local systems and higher {T}eichm\"{u}ller theory.
\newblock {\em Publ. Math. Inst. Hautes \'{E}tudes Sci.}, (103):1--211, 2006.

\bibitem{GdlH90}
\'Etienne Ghys and Pierre~de~la Harpe, editors.
\newblock ``Sur les groupes hyperboliques d'apr\`es {M}ikhael {G}romov,''
  volume~83 of {\em Progress in Mathematics}.
\newblock Birkh\"auser Boston, Inc., Boston, MA, 1990.
\newblock Papers from the Swiss Seminar on Hyperbolic Groups held in Bern,
  1988.

\bibitem{GITIK199965}
Rita Gitik.
\newblock Ping-pong on negatively curved groups.
\newblock {\em J. Algebra}, 217(1):65--72, 1999.

\bibitem{GWtheta}
Olivier Guichard and Anna Wienhard.
\newblock Positivity and higher Teichm\"uller theory.
\newblock ArXiv:1802.02833, 2018.

\bibitem{GWthetaTwo}
Olivier Guichard and Anna Wienhard.
\newblock Generalizing Lusztig's total positivity.
\newblock ArXiv:2210.11605, 2022.

\bibitem{Guichard:2021aa}
Olivier Guichard, Fran{\c c}ois Labourie, and Anna Wienhard.
\newblock Positivity and representations of surface groups.
\newblock ArXiv:2106.14584, 2021.

\bibitem{Ivascu}
Dumitru Ivascu.
\newblock On {K}lein--{M}askit combination theorem.
\newblock In ``Romanian--{F}innish Seminar on Complex Analysis,'' volume 743,
  pages 115--124. Springer Lecture Notes in Mathematics, 1976.


\bibitem{MR2553578}
Michael Kapovich, ``Hyperbolic manifolds and discrete groups,'' 
Mod. Birkh\"auser Classics. Birkh\"auser Boston, Ltd., Boston, MA, 2009, xxviii+467 pp.



\bibitem{MR3811766}
Michael Kapovich and Bernhard Leeb.
\newblock Finsler bordifications of symmetric and certain locally symmetric
  spaces.
\newblock {\em Geom. Topol.}, 22(5):2533--2646, 2018.

\bibitem{MR3736790}
Michael Kapovich, Bernhard Leeb, and Joan Porti.
\newblock Anosov subgroups: dynamical and geometric characterizations.
\newblock {\em Eur. J. Math.}, 3(4):808--898, 2017.

\bibitem{KLP:Morse}
Michael Kapovich, Bernhard Leeb, and Joan Porti.
\newblock A {M}orse lemma for quasigeodesics in symmetric spaces and
  {E}uclidean buildings.
\newblock {\em Geom. Topol.}, 22(7):3827--3923, 2018.

\bibitem{Kapovich-Sardar}
Michael Kapovich and Pranab Sardar.
\newblock ``Trees of hyperbolic spaces,''
\newblock volume 282 of {\em Mathematical Surveys and Monographs}. American Mathematical Society, Providence, RI, 2024.

\bibitem{MR1390041}
Olga Kharlampovich and Alexei Myasnikov.
\newblock Hyperbolic groups and free constructions.
\newblock {\em Trans. Amer. Math. Soc.}, 350(2):571--613, 1998.

\bibitem{klein1883neue}
Felix Klein.
\newblock Neue {B}eitr\"{a}ge zur {R}iemann'schen {F}unctionentheorie.
\newblock {\em Math. Ann.}, 21(2):141--218, 1883.

\bibitem{MR2221137}
Fran\c{c}ois Labourie.
\newblock Anosov flows, surface groups and curves in projective space.
\newblock {\em Invent. Math.}, 165(1):51--114, 2006.

\bibitem{LOW1}
Liulan Li, Ken'ichi Ohshika, and Xiantao Wang.
\newblock On {K}lein--{M}askit combination theorem in space, {I}.
\newblock {\em Osaka J. Math.}, 46:1097--1141, 2009.

\bibitem{LOW2}
Liulan Li, Ken'ichi Ohshika, and Xiantao Wang.
\newblock On {K}lein-{M}askit combination theorem in space {II}.
\newblock {\em Kodai Math. J.}, 38(1):1--22, 2015.

\bibitem{Lyndon-Schupp}
Roger~C. Lyndon and Paul~E. Schupp.
\newblock ``Combinatorial group theory.''
\newblock Classics in Mathematics. Springer-Verlag, Berlin, 2001.
\newblock Reprint of the 1977 edition.

\bibitem{pedroza2008combination}
Eduardo Mart\'{\i}nez-Pedroza.
\newblock {\em Combination of quasiconvex subgroups in relatively hyperbolic
  groups}.
\newblock ProQuest LLC, Ann Arbor, MI, 2008.
\newblock Thesis (Ph.D.)--The University of Oklahoma.

\bibitem{pedroza2009}
Eduardo Mart\'{\i}nez-Pedroza.
\newblock Combination of quasiconvex subgroups of relatively hyperbolic groups.
\newblock {\em Groups Geom. Dyn.}, 3(2):317--342, 2009.

\bibitem{MR2994828}
Eduardo Mart\'{\i}nez-Pedroza and Alessandro Sisto.
\newblock Virtual amalgamation of relatively quasiconvex subgroups.
\newblock {\em Algebr. Geom. Topol.}, 12(4):1993--2002, 2012.

\bibitem{maskit1965klein}
Bernard Maskit.
\newblock On {K}lein's combination theorem.
\newblock {\em Trans. Amer. Math. Soc.}, 120:499--509, 1965.

\bibitem{maskit1968klein}
Bernard Maskit.
\newblock On {K}lein's combination theorem. {II}.
\newblock {\em Trans. Amer. Math. Soc.}, 131:32--39, 1968.

\bibitem{maskit1971klein}
Bernard Maskit.
\newblock On {K}lein's combination theorem. {III}.
\newblock In ``Advances in the {T}heory of {R}iemann {S}urfaces ({P}roc.
  {C}onf., {S}tony {B}rook, {N}.{Y}., 1969),'' Ann. of Math. Studies, No. 66,
  pages 297--316. Princeton Univ. Press, Princeton, N.J., 1971.

\bibitem{maskit:book}
Bernard Maskit.
\newblock ``Kleinian groups,'' volume 287 of {\em Grundlehren der
  mathematischen Wissenschaften}.
\newblock Springer-Verlag, Berlin, 1988.

\bibitem{maskit1993klein}
Bernard Maskit.
\newblock On {K}lein's combination theorem. {IV}.
\newblock {\em Trans. Amer. Math. Soc.}, 336(1):265--294, 1993.


\bibitem{MM} 
Ashot Minasyan and Lawk Mineh
\newblock Quasiconvexity of virtual joins and separability of products in relatively hyperbolic groups.
\newblock ArXiv: 2207.03362, 2022. 


\bibitem{Mitra1998}
Mahan Mitra.
\newblock Coarse extrinsic geometry: a survey.
\newblock In ``The {E}pstein birthday schrift,'' volume~1 of {\em Geom.
  Topol. Monogr.}, pages 341--364. Geom. Topol. Publ., Coventry, 1998.

\bibitem{MR4472055}
Mahan Mj and Sabyasachi Mukherjee.
\newblock Combination theorems in groups, geometry and dynamics.
\newblock In ``In the tradition of {T}hurston {II}. {G}eometry and groups,''
  pages 331--383. Springer, 2022.
  
\bibitem{morgan84}
John~W. Morgan.
\newblock On {T}hurston's uniformization theorem for three-dimensional
  manifolds.
\newblock In ``The {S}mith conjecture ({N}ew {Y}ork, 1979),'' volume 112 of
  {\em Pure Appl. Math.}, pages 37--125. Academic Press, Orlando, FL, 1984.

\bibitem{MR1677888} 
 Jean-Pierre Otal, {\em Thurston's hyperbolization of Haken manifolds.} In: 
 ``Surveys in differential geometry,'' Vol. III (Cambridge, MA, 1996), pp. 77--194.
International Press, Boston, MA, 1998. 

\bibitem{trees}
Jean-Pierre Serre.
\newblock ``Trees.'' 
\newblock Springer Monographs in Mathematics. Springer-Verlag, Berlin, 2003.
\newblock Translated from the French original by John Stillwell, Corrected 2nd
  printing of the 1980 English translation.
  
\bibitem{TW}
Alec Traaseth and Theodore Weisman.
\newblock Combination theorems for geometrically finite convergence groups.
\newblock ArXiv:2305.08011, 2023,

\end{thebibliography}
\end{document}